\documentclass[english]{article}

\newif\ifworkinprogress
\workinprogresstrue

\ifworkinprogress
  \newcommand{\RM}[1]{\textcolor{blue}{\textbf{[RM] #1}}}
  \newcommand{\HW}[1]{\textcolor{green}{\textbf{[HW] #1}}} 		  
  \newcommand{\SL}[1]{\textcolor{red}{\textbf{[SL] #1}}}
\else
  \newcommand{\RM}[1]{}
  \newcommand{\HW}[1]{}
  \newcommand{\SL}[1]{}
\fi

\usepackage[numbers]{natbib}
\usepackage{definitions}
\usepackage[colorlinks,linkcolor=blue]{hyperref}
\usepackage{color}
\newcommand{\R}{\mathbb{R}}
\newcommand{\T}{\top}
\newcommand{\la}{\langle}
\newcommand{\ra}{\rangle}
\newcommand{\dom}{{\rm dom}}

\newcommand{\rminf}{\mathop{{\rm inf}}}
\newcommand{\rmsup}{\mathop{{\rm sup}}}
\newcommand{\lam}{\lambda}
\newcommand{\dt}{\delta}
\newcommand{\Dt}{\Delta}
\newcommand{\ran}{{\rm range}}

\newcommand{\al}{\alpha}

\newcommand{\ga}{\gamma}
\newcommand{\ep}{\epsilon}

\newcommand{\na}{\nabla}
\newcommand{\lt}{\left}
\newcommand{\rt}{\right}

\newcommand{\argmax}{\mathop{{\rm argmax}}}
\newcommand{\argmin}{\mathop{{\rm argmin}}}
\newcommand{\spn}{{\rm span}}
\newcommand{\bsone}{\boldsymbol{1}}
\newcommand{\diam}{{\rm diam}}

\newcommand{\stt}{{\rm s.t.}}
\newcommand{\td}{\tilde}
\newcommand{\wtd}{\widetilde}
\newcommand{\mc}{\mathcal}
\newcommand{\Om}{\Omega}

\newcommand{\cE}{\mathcal{E}}

\newcommand{\cN}{\mathcal{N}}

\newcommand{\cP}{\mathcal{P}}

\newtheorem{theorem}{Theorem}[section]

\newtheorem{lemma}[theorem]{Lemma}
\newtheorem{proposition}[theorem]{Proposition}

\newtheorem{definition}{Definition}[section]

\newtheorem{remark}[theorem]{Remark}
\newtheorem{assumption}[theorem]{Assumption}

\begin{document}
\title{Frank-Wolfe Methods with an Unbounded Feasible Region and Applications to Structured Learning} 
\author{Haoyue Wang\thanks{MIT Operations Research Center (email: haoyuew@mit.edu).} \and
Haihao Lu\thanks{Booth School of Business, University of Chicago
({email:  haihao.lu@chicagobooth.edu}).}
\and
Rahul Mazumder\thanks{MIT Sloan School of Management, Operations Research Center and MIT Center for Statistics ({email: rahulmaz@mit.edu}).
This research was partially supported by awards from
the Office of Naval Research ONR-N000141812298 (Young Investigator Award), the National Science Foundation (NSF-IIS-1718258), MIT-IBM Watson
AI Lab to Rahul Mazumder. 
}}
\date{}
\maketitle

\begin{abstract}
    The Frank-Wolfe (FW) method is a popular algorithm for solving  large-scale convex optimization problems appearing in structured statistical learning. 
    However, the traditional Frank-Wolfe method can only be applied when the feasible region is bounded, which limits its applicability in practice. Motivated by two applications in statistical  learning, the $\ell_1$ trend filtering problem and matrix optimization problems with generalized nuclear norm constraints, we study a family of convex optimization problems where the unbounded feasible region is the direct sum of an unbounded linear subspace and a bounded constraint set. We propose two 
    new Frank-Wolfe methods: unbounded Frank-Wolfe method (uFW) and unbounded Away-Step Frank-Wolfe method (uAFW), for solving a family of convex optimization problems with this class of unbounded feasible regions. We show that under proper regularity conditions, the unbounded Frank-Wolfe method has a $O(1/k)$ sublinear convergence rate, and unbounded Away-Step Frank-Wolfe method has a linear convergence rate, matching the best-known results for the Frank-Wolfe method when the feasible region is bounded. Furthermore, computational experiments indicate that our proposed methods appear to outperform alternative solvers. 
\end{abstract}

\section{Introduction}\label{sec:intro}

The Frank-Wolfe (FW) method \cite{FW-origin}, also known as the conditional gradient method \cite{Demyanov1970approximate}, is a well-studied first-order algorithm for smooth convex optimization with a bounded feasible region. 
Compared to other first-order methods, such as projected gradient methods and proximal type methods \cite{Parikh2014proximal}, where a projection operation onto the feasible set is required at every iteration, the Frank-Wolfe method avoids projection by minimizing a linear objective over the feasible set. Solving this problem is often computationally more attractive than a projection step in several large-scale problems arising in machine learning. For instance, when the constraint set $S$ is polyhedral, the linear subproblem is given by a linear program, which may be a computationally friendlier alternative  compared to 
solving a convex quadratic program in the projection step. When the constraint set $S$ is a nuclear norm ball, the solution to the linear subproblem can be obtained by computing the leading singular vector/value pair, while the projection onto $S$ may require computing several leading singular vectors/values. Due to its computational efficiency and projection-free nature, the Frank-Wolfe method has emerged as a popular choice for solving large-scale convex optimization problems arising in machine learning applications~\cite{jaggi2013revisiting}.

In many real-world applications however, the feasible region of the optimization problem may be unbounded, which limits the applicability of Frank-Wolfe methods. We present two motivating examples from statistical  learning/high-dimensional statistics.
In the generalized lasso problem \cite{gen-lasso}, the problem of interest is to solve
\begin{equation}\label{intro:generalized-lasso}
	\min_{x\in \R^n} \|b - A x\|_2^2 \ ~~~\stt~~~ \|Hx\|_1 \le \delta \ ,
\end{equation}
where $x\in \R^n$ are the model coefficients or signal of interest (decision variable), $A\in \R^{N\times n}$ is the model matrix, $b\in \R^N$ is the response,  and $H\in \R^{m\times n}$ is a general (usually non-square or singular) matrix imposing additional structure on the unknown signal $x$. A special case is the  $\ell_1$ trend filtering problem \cite{fused-lasso,Ryan-trend-filter,Boyd-trend-filter}, where the matrix $H$ is the $r$-th order discrete derivative matrix with $r\ge 1$.
In generalized nuclear norm regularization problems \cite{Fithian2018fexible}, we are interested in the following problem
\begin{equation}\label{intro:generalized-nuclear-norm}
	\min_{X\in \R^{m\times n}} f(X) ~~\stt~~ \|PXQ\|_* \le \delta \ 
\end{equation}
where, the objective $f$ is a smooth convex function of the matrix variable $X\in \R^{m\times n}$, i.e., the gradient $\na f$ is Lipschitz continuous; $\|\cdot\|_*$ denotes the nuclear norm (i.e. the sum of singular values) of a matrix, and $P,Q$ are general (usually singular) matrices.  In example~\eqref{intro:generalized-nuclear-norm}, the feasible region is unbounded when matrices $P$ and $ Q$ are not full rank. Thus, the traditional Frank-Wolfe method is no longer applicable for these problems, even though the Frank-Wolfe method is known to work well for optimization problems involving the 
$\ell_1$ norm ball or nuclear norm ball constraints \cite{hazan2008sparse,Freund-inface}.

Notice that in Problems \eqref{intro:generalized-lasso} and \eqref{intro:generalized-nuclear-norm},
the feasible regions can be expressed as the direct sum of a linear subspace $T$ and a bounded set $S$. For example, the constraint set $\{x~|~\|Hx\|_1\le \delta\}$ in \eqref{intro:generalized-lasso} can be formulated as $T \oplus S:=\{x~|~x=s+t, t\in T, s\in S\}$, where $T= \ker(H)$ is a linear subspace, and $S=\{x\in \R^n~|~\|Hx\|_1 \le \delta, x \in (\ker(H) )^{\perp}\}$ is a bounded set in $\R^n$. The constraint $\{X~|~\|PXQ\|_*\le \delta\}$ in \eqref{intro:generalized-nuclear-norm} can also be expressed as $T \oplus S$ with a linear subspace $T$ and a bounded set $S$ (see Section \ref{subsection:generalized-nuclear-norm} for details). 
Thus motivated, in this paper, we study the following smooth constrained convex optimization problem over an unbounded feasible region $T \oplus S$:
\begin{equation}\label{problem1}
	\min_{x\in \R^n} f(x)~~~\stt~~~x \in T \oplus S
\end{equation}
where $f$ is a smooth convex function on $\R^n$, $T$ is a linear subspace of $\R^n$, $S$ is a bounded convex set ($S\subset T^{\perp}$), and $\oplus$ is the direct sum of two orthogonal spaces.

The major contribution of this paper is to generalize the traditional Frank-Wolfe method to solve~\eqref{problem1} with computational guarantees. In Section~\ref{sec:algorithm}, we present two new algorithms designed for \eqref{problem1}: {the unbounded Frank-Wolfe Method (uFW)} and {the unbounded Away-Step Frank-Wolfe Method (uAFW)}. We introduce new curvature constants generalizing those arising in the analysis of Frank-Wolfe methods over a bounded feasible region.
Our key idea is to alternate between a Frank-Wolfe step along the bounded set $S$ where we solve a linear subproblem; and a gradient descent step along the unbounded linear subspace $T$ which admits a simple projection operation. In Section~\ref{section: computational guarantees}, we show that under suitable conditions, 
uFW converges to an optimal solution of~\eqref{problem1} with a sublinear-rate $O(1/k)$, and 
uAFW converges to an optimal solution of~\eqref{problem1} with a linear rate. In Section~\ref{section: applications}, we discuss how to apply our proposed algorithms to 
Problems~\eqref{intro:generalized-lasso} and \eqref{intro:generalized-nuclear-norm},
focusing on how to make the computations efficient (by exploiting problem structure). Section~\ref{section: experiments} presents numerical experiments suggesting that our proposal outperforms alternatives by a significant margin.

\subsection{Related literature}\label{subsection: intro-related works}
As mentioned earlier, the Frank-Wolfe method has been extensively studied in the optimization community. The original Frank-Wolfe method, dating back to Frank and Wolfe~\cite{FW-origin},  was designed to solve a smooth convex optimization problem over a polytope. The method was then extended to more general settings with a convex bounded feasible region, and was shown to attain an $O(1/k)$ sublinear rate of convergence~\cite{dunn1978conditional,dunn1979rates,polyak1987introduction}.
Indeed, the $O(1/k)$ sublinear rate matches the lower complexity bound for solving a generic constrained convex optimization problem~\cite{lan2013complexity}. Recently, due to its projection-free nature, there has been renewed interest in the Frank-Wolfe method for solving large-scale optimization problems arising in structured statistical learning. Compared to other first-order methods (e.g., proximal gradient), which usually require us to  compute a projection onto the constraint set (at every iteration), the Frank-Wolfe method solves a linear subproblem, which can be simpler than the projection problem in many applications~\cite{jaggi2013revisiting, harchaoui2015conditional, hazan2008sparse, clarkson2010coresets}. 

Recent works have explored different properties of the FW method and its variants.
For example,~\cite{jaggi2013revisiting} studies the invariance of the Frank-Wolfe method under linear transformations, and introduces a curvature constant that improves the complexity analysis; 
\cite{lan2016conditional} proposes gradient sliding schemes which achieve the lower complexity bound on both the number of linear optimization oracle calls as well as gradient computations; 
\cite{bomze2019first} proves support identification in finite time of two variants of the FW method on the standard simplex, and later work \cite{bomze2020active} develops complexity bounds for support identification.
See for example, \cite[Chapter 7]{lan2020first} for a nice overview on the recent developments of the Frank-Wolfe method.

	Recently, \cite{gonccalves2020projection} proposes an interesting Frank-Wolfe type method to address problems with unbounded constraint sets, but their approach differs from what we propose. \cite{gonccalves2020projection} use accelerated gradient descent as their base algorithm and use FW to approximately solve the projection step. 
	This is a two-loop algorithm, which behaves similar to accelerated gradient descent.
	In contrast, our method is a single-loop algorithm, and resembles a Frank-Wolfe method.
	To deal with the unbounded constraint in each subproblem, \cite{gonccalves2020projection} performs a Frank-Wolfe step on the intersection of $S\oplus T$ and an Euclidean ball centered at the initial point. 
	This paper reports complexities of $O(1/\sqrt{\ep})$ for the number of gradient evaluations and $O(1/\ep)$ for the number of linear oracles in order to find a solution with primal optimality gap $\ep$. 
	Although their complexity bound on the number of gradient evaluations is better, each gradient evaluation is accompanied by solving many linear oracles. While the number of linear oracles is the same rate as ours, the linear oracles are different (due to the presence of an additional ball constraint) and can be harder to solve than the ones arising in our FW procedure.
	In addition, for the setting with a strongly convex objective function and polyhedral constraints, we prove linear convergence results---similar results are not discussed in~\cite{gonccalves2020projection}. 
	Our numerical experiments (cf. Section \ref{section: experiments}) appear to suggest that our proposal is computationally more attractive compared to the proposal of~\cite{gonccalves2020projection}.

The Away-step Frank-Wolfe method (AFW) is a variant of FW first proposed by Wolfe in 1970s~\cite{wolfe1970convergence}, that attempts to ameliorate the so-called zigzagging behavior of the vanilla FW method. 
Later,~\cite{guelat1986some} proposes a modified AFW that has a linear convergence guarantee for strongly convex objective function and under a strict complementarity condition. 
	Recently, \cite{garber2016linearly,lacoste2015global} propose new variants of AFW which are linearly convergent for strongly convex objective function on a polytope. The version of AFW in \cite{lacoste2015global} has been further investigated by \cite{beck2017linearly,pena2016neumann,pena2019polytope} with a different analysis technique.  
	In this paper, we generalize the version of AFW in \cite{lacoste2015global} to the structured unbounded setting \eqref{problem1} and prove a linear convergence rate when the objective function is strongly convex and the constraint set is a polyhedron.

The efficacy of the Frank-Wolfe method depends upon how easily one can solve the linear optimization oracle. A recent survey
\cite{jaggi2013revisiting} lists several constraint sets arising in modern applications that admit efficient linear optimization oracles. Examples include the $\ell_1$-norm ball~\cite{clarkson2010coresets} and the nuclear norm ball~\cite{hazan2008sparse,Freund-inface}, among others. Here we show that Frank-Wolfe type methods can also be applied to problems \eqref{intro:generalized-lasso} and \eqref{intro:generalized-nuclear-norm}, where the constraint sets are transformed/unbounded versions of the $\ell_1$-norm ball and nuclear norm ball, respectively.

A special case of~\eqref{intro:generalized-lasso} is the $\ell_1$ trend filtering problem~\cite{Boyd-trend-filter}. Statistical and computational aspects of this problem have been studied in~\cite{Boyd-trend-filter,Ryan-trend-filter,gen-lasso,wahlberg2012admm}. 
Currently, popular algorithms for the $\ell_1$ trend filtering problem include interior point methods \cite{Boyd-trend-filter},
ADMM-based algorithms
\cite{wahlberg2012admm,ramdas2016fast}, and parametric quadratic programming \cite{gen-lasso}.
The generalized nuclear norm regularized matrix estimation problem~\eqref{intro:generalized-nuclear-norm},
has important applications in computer vision~\cite{angst2011generalized}, collaborative filtering~\cite{srebro2010collaborative}, and matrix completion with side information and 
missing data~\cite{Fithian2018fexible,eftekhari2018weighted,chiang2015matrix}. 
To our knowledge,
Frank-Wolfe type methods are yet to be explored for problems~\eqref{intro:generalized-lasso} and~\eqref{intro:generalized-nuclear-norm}---bridging this gap is the focus of our paper.

\smallskip

\noindent {\bf{Notation:}} We let $\R^n$ denote the $n$-dimensional Euclidean space, and $\R^{m\times n}$ the space of $m\times n$ real matrices. For $x,y\in \R^n$, $\la x,y\ra =\sum_{i=1}^n x_i y_i$ denotes the Euclidean inner product of $x$ and $y$. Let $\bsone_n$ denote the vector in $\R^n$ with all coordinates being $1$. Let $\boldsymbol{0}_n$ denote the vector in $\R^n$ with all coordinates being $0$, and $ \boldsymbol{0}_{m\times n} $ be the matrix in $\R^{m\times n}$ with all elements being $0$.
Let $e_i$ be the unit vector in $\R^n$ with the $i$-th coordinate being $1$ and all other coordinates being $0$. Let $I_n$ be the identity matrix in $\R^{n\times n}$.
For $A,B \in \R^{m\times n}$, $\la A, B \ra:= \sum_{i=1}^m \sum_{j=1}^n A_{i,j} B_{i,j} $ defines the inner product of $A$ and $B$, and $A \otimes B$ denotes their Kronecker product. For a function $f$, $\dom(f)$ denotes the domain of $f$.
For a given linear subspace $T\subset\R^n$, we let $\dim(T)$ denote the dimension of $T$, and $T^\perp$ its orthogonal subspace. Furthermore, $\cP_T$ and $\cP_T^{\perp}$ denote the projection
operators onto the linear subspaces $T$ and its orthogonal subspace $T^{\perp}$ (respectively). For two subsets $T,S\subset \R^n$, $T\oplus S$ denotes the direct sum of the two subsets, i.e., $T\oplus S= \{x+y~|~  x\in T, y\in S\}$.
A function $f$ is said to be $L$-smooth over a set $Q$ if it is differentiable and $\|\nabla f(u) - \nabla f (v)\|_2 \leq L \| u - v\|_2$ for all $u,v \in Q$.

\section{Unbounded Frank-Wolfe Algorithms}\label{sec:algorithm}
In this section, we present two new Frank-Wolfe algorithms (with and without away-step) designed to solve \eqref{problem1} with an unbounded feasible region. As mentioned above, we alternate between performing a Frank-Wolfe step on the bounded set $S$ and a gradient descent step along the subspace $T$. Since the gradient descent direction can also be viewed as an extreme ray to the solution of the Frank-Wolfe linear subproblem along the linear subspace $T$, we call this the ``Unbounded (Away Step) Frank-Wolfe Method''. 

Algorithm \ref{CG-G method} presents our first algorithm, the unbounded Frank-Wolfe method, for solving \eqref{problem1}. In each iteration, we first perform a gradient descent step along the unbounded subspace $T$. To do so, we first compute the negative projected gradient onto $T$, namely $-\cP_{T}\nabla f(x^{k})$, and move the current solution towards this direction with step-size $\eta$. While the projection onto a bounded set (e.g., $S$) can be expensive, the projection onto a linear subspace is usually much cheaper, and the solution can be computed in closed-form. Moreover, for $\ell_1$ trend filtering, the subspace $T$ is in fairly low dimension, making the projection step even simpler. 
Section \ref{section: applications} shows that a solution $s^k$ of the linear subproblem within the bounded set $S$, can be solved efficiently for these two motivating examples by making use of problem structure. Finally, we perform the Frank-Wolfe step within set $S$ using step-size $\alpha_k$ by moving along the direction $s^k-\cP_T^{\perp} x^k$, i.e., the direction between the resulting extreme point and the current solution projected onto set $S$ (note that $\cP_{T}^\perp x^k = \cP_T^\perp y^k$ as $\cP_T^\perp \cP_T = 0$). 
Different step-size rules have been explored in the literature of the Frank-Wolfe method~\cite{Freund-2016-analysis}. Here, we study two step-size rules:
\begin{itemize}
	\item Line search step-size rule: 
	\begin{equation}\label{line-search step-size rule}
		\alpha_{k} \in \argmin_{\al\in[0,1]} f(y^k + \al (s^k - \cP_T^{\perp}x^{k}) ).
	\end{equation}
	
	\item Simple step-size rule: 
	\begin{equation}\label{simple step-size rule}
		\al_k = \tfrac{2}{k+2} ~ \text{if}~ f(y^k
		+\tfrac{2}{k+2}(s^{k}-\cP_T^{\perp} x^k ))\le f(x^0); ~\al_k = 0~ \text{otherwise.}
	\end{equation}
\end{itemize}

The simple step-size rule avoids line search. Indeed, $\al_k = {2}/{(k+2)}$ is the standard step-size in the Frank-Wolfe literature~\cite{Freund-2016-analysis}; and here, for technical reasons, we also want to make sure that $x^k$ stays in a level set of $f$, i.e., $f(x^k)\le f(x^0)$ for $k\ge 1$. In practice, this condition is always satisfied after the first several iterations.

\begin{algorithm}
	\caption{Unbounded Frank-Wolfe Method (uFW)}
	\label{CG-G method}
	\begin{algorithmic}
		\STATE \textbf{Initialize.} Initialize with $x^{0}\in T \oplus S$, gradient descent step-size $\eta$ and Frank-Wolfe step-size sequence $\alpha_k\in [0,1]$ for all $k$.
		\STATE {Perform the following steps for iterations $k=0,1,2,...$}
		\STATE \ \ \ \ \ \ \ \ \ (1) \textbf{Gradient descent step} (along subspace $T$):\ \  $y^{k}=x^{k}-\eta\cP_{T}\nabla f(x^{k})$.
		\STATE \ \ \ \ \ \ \ \ \ (2) \textbf{Solve linear oracle}:\ \  $s^{k}=\arg\min_{s\in S}\langle\nabla f(y^{k}),s\rangle$.
		\STATE \ \ \ \ \ \ \ \ \ (3) \textbf{Frank-Wolfe step} (along subset $S$):\ \  $x^{k+1}= y^k
		+\alpha_{k}(s^{k}-\cP_T^{\perp} x^k )$.
	\end{algorithmic}
\end{algorithm}
Unlike gradient descent, the traditional Frank-Wolfe method does not have linear convergence even if the objective function is strongly convex. An intuitive explanation is that the solutions to the linear subproblems in the Frank-Wolfe algorithm may alternate between two extreme points of the constraint set, the iterate solutions zigzag, slowing down the convergence of the algorithm (see~\cite{lacoste2015global} for a nice explanation). Away-step Frank-Wolfe method allows moving in a direction opposite to the maximal solution of the linear model evaluated at the current solution. The iterate solution can land on a certain face of the constraint set, thereby avoiding the aforementioned zigzagging phenomenon. Moreover, with away-steps, the Frank-Wolfe method enjoys a linear convergence rate when the objective function is further strongly convex and the constraint set is a polytope \cite{lacoste2015global,pena2016neumann,pena2019polytope}. In addition to faster convergence, the away-step Frank-Wolfe method can usually lead to a sparser solution \cite{bomze2019first,bomze2020active}, which is helpful in many applications where sparsity and interpretability are desirable properties of a solution~\cite{Freund-inface}.

Algorithm \ref{CG-G with away steps} adapts the away-step Frank-Wolfe method to solve \eqref{problem1} with the unbounded feasible region $T\oplus S$. The major difference of Algorithm \ref{CG-G with away steps} with Algorithm \ref{CG-G method} is that the former allows performing an away step for the update along $S$. 
Note that $\cP_T^\perp x^k$ can be represented as a convex combination of at most $k$ vertices $V(x^k)\subset \text{vertices}(S)$, which can be written as $\cP_T^\perp x^k = \sum_{v\in V(x^k)} \lam_{v}(x^k) v$, with $\lam_{v}(x^k)>0$ for $v\in V(x^k)$ and $ \sum_{v\in V(x^k)} \lam_{v}(x^k) = 1$. In Algorithm \ref{CG-G with away steps}, we keep track of the set of vertices $V(x^k)$ and the weight vector $\lam(x^k)$, and choose between the Frank-Wolfe step and away step accordingly. After a step is taken in the direction $T^\perp$, we update the set of vertices $V(x^{k+1})$ and the weights $\lam(x^{k+1})$. Unlike Algorithm \ref{CG-G method}, we perform a line search to find the Frank-Wolfe
step-size\footnote{Note that other step size rules using the smoothness parameter or backtracking line search \cite{pedregosa2020linearly} also work for uAFW under which linear convergence can be proved (with slightly different arguments). We focus on the exact line-search steps for simplicity.} (the simple step-size rule does not apply here). Note the cardinality of the vertex set $V(x^k)$ in the maximization problem in Step (2) of Algorithm \ref{CG-G with away steps} is usually small, so the linear oracle for computing $v^k$ can be much simpler than the computation of $s^k$.

Formally speaking, if we decide to take a Frank-Wolfe step (See Algorithm \ref{CG-G with away steps}), we update the vertex set by 
\begin{eqnarray}\label{update-vertex1}
	V(x^{k+1}) = \{s^k\} ~~ \text{if} ~ \al_k=1;~ \text{and} ~~ V(x^{k+1}) = V(x^k) \cup \{s^k\} ~~ \text{if} ~ \al_k\neq1.
\end{eqnarray}
We update the weights by
\begin{equation}\label{update-weight1}
	\begin{aligned}
		&\lam_{s^k}(x^{k+1}) = (1-\al_k) \lam_{s^k} (x^k) + \al_k, \\
		\text{and}~~ & \lam_v(x^{k+1}) = (1-\al_k) \lam_v (x^{k})~~~~~\text{for}~~v\in V(x^k) \backslash \{s^k\}.
	\end{aligned}
\end{equation}
\begin{equation}\label{update-vertex2}
	V(x^{k+1}) = V(x^{k}) \backslash \{v^k\} ~ \text{if} ~ \al_k = \al_{\max};~~\text{and}~~ V(x^{k+1}) = V(x^k) ~ \text{if} ~ \al_k \neq \al_{\max}.
\end{equation}
We update the weights by 
\begin{equation}\label{update-weight2}
	\begin{aligned}
		&\lam_{v^k} (x^{k+1}) = (1+ \al_k) \lam_{v^k} (x^k) - \al_k, \\
		\text{and} ~~ & \lam_{v}(x^{k+1}) = (1+\al_k) \lam_v^{k} ~~~~~ \text{for} ~  v\in V(x^k) \backslash \{v^k\}.
	\end{aligned}
\end{equation}
This update follows from recent work on away-step Frank-Wolfe methods~\cite{lacoste2015global, pena2019polytope}.

\begin{algorithm}
	\caption{Unbounded Away-Step Frank-Wolfe Method (uAFW)}
	\label{CG-G with away steps}
	\begin{algorithmic}
		\STATE \textbf{Initialize.} Initialize with $x^{0}\in T \oplus S$ such that $\cP_T^{\perp} x^0 $ is a vertex of $S$, vertex set $V(x^0) =\{ \cP_T^\perp x^0 \}$, weight parameters $\lam_{v}(x^0) = 1$ for $v\in V(x^0)$, and gradient descent step-size $\eta$.
		\STATE Perform the following updates for iterations $k=0,1,2, \dots$:
		\STATE \ (1) {Gradient descent} (along subspace $T$):
		$
		y^{k}=x^{k}-\eta \cP_{T}\nabla f(x^{k}) . $
		\STATE \  (2) {Linear oracle}: 
		$
		s^{k}:= \arg\min_{s\in S}\la{\na f(y^{k})},s\ra, ~ 
		v^k:= \argmax_{v\in V(x^k)} \la {\na f(y^{k})},v \ra .
		$
		\STATE \ (3) {Find the moving direction}: \\
		\ \ \    \quad \textbf{if} $\la \na f(y^k), s^k- \cP_T^{\perp} x^k \ra < \la \na f(y^k), \cP_T^{\perp} x^k  - v^k \ra  $: (Frank-Wolfe step)
		\STATE \ \ \   \quad \quad $d^k:= s^k - \cP_T^{\perp} x^k$; $\al_{\max} = 1$.
		\STATE \ \ \   \quad \textbf{else}: (away step)
		\STATE \ \ \  \quad \quad $d^k:= \cP_T^{\perp}x^k - v^k$; $\al_{\max} = {\lam_{v^k}(x^k)  }/({1- \lam_{v^k}(x^k)}).$
		\STATE \   
		(4) {FW/away step} (along subset $S$):  $x^{k+1} = y^{k}+\alpha_{k}d^k$,
		where $\alpha_k\in[0,\al_{\max}]$ is chosen by line search.
		\STATE \   
		(5) Update vertices $V(x^{k+1})$ and weights  $\lam_v(x^{k+1})$ for $v\in V(x^{k+1})$ according to equations \eqref{update-vertex1} - \eqref{update-weight2}.
	\end{algorithmic}
\end{algorithm}

\section{Computational Guarantees}\label{section: computational guarantees}
In this section, we derive computational guarantees of uFW and uAFW. When the objective is smooth, we show that uFW converges to an optimal solution with sublinear rate $O(1/k)$. Additionally, if the objective is strongly convex and the constraint set $S$ is a polyhedron, we show that uAFW converges linearly to an optimal solution.

\subsection{Sublinear Rate for uFW}\label{subsection: analysis of uFW}

We start by introducing some quantities that appear in the computational guarantees of uFW.

Let $X^0$ denote the level set of $f(x)$ with maximal value $f(x^0)$, that is, 
\begin{eqnarray}\label{def:X0}
	X^0 := \{x\in \R^n ~|~ f(x) \le f(x^0)\}.
\end{eqnarray}
For the two step-size rules \eqref{simple step-size rule} and \eqref{line-search step-size rule}, the iterations $x^k$ and $y^k$ stay in the level set 
$X^0$ (see analysis below).
Throughout the paper, we assume that $f$ has a bounded level set along the subspace $T$, i.e.,
\begin{equation}\label{def: DT}
	D_T:= \mathop{\rmsup}\limits_{x,y\in X^0 } \|\cP_{T}(x-y)\|_2 < \infty \ .
\end{equation}
The traditional Frank-Wolfe method is invariant under affine transformations.  
\cite{jaggi2013revisiting}~introduces a curvature constant of the objective with respect to the bounded constraint set, which is used to improve the convergence rate of the Frank-Wolfe method by using the  invariance (of affine transformations). In order to measure the progress of the Frank-Wolfe step in Algorithm \ref{CG-G method}, we extend the curvature constant in \cite{jaggi2013revisiting} to our setting \eqref{problem1} with an unbounded constraint set:

\begin{definition}
	The curvature constant of $f$ with respect to the constraint set $S$ and level set $X^0$ is defined as
	\begin{equation}\label{def: Cf}
		\begin{aligned}
			C_{f,x^0}^S:=
			\sup_{s\in S ,\ x,y\in \R^n, \al\in (0,1]}
			\big\{ &
			({2}/{\al^2})\lt[  f(y) - f(x) - \la  \na f(x), y-x \ra  \rt] ~\big| \\
			& \cP_T^{\perp} x \in S,\  x\in X^0,\ 
			y = x+ \al (s - \cP_T^\perp x) \big\} \ .
		\end{aligned}
	\end{equation}
\end{definition}

The parameter $C_{f,x^0}^S$ quantifies the curvature information of $f$ when moving within the set $S$, and its value is dependent on the function $f$, the constraint set $S\oplus T$, and the initial point $x^0$. In particular, it is easy to check that $C_{f,x^0}^S$ is smaller than $L'\diam(S)^2$, where $L'$ is the smoothness parameter of $f$ with respect to any norm, and $\diam(S)$ is the diameter of set $S$ with respect to the same chosen norm. A related quantity appears in the standard  analysis of the Frank-Wolfe method \cite{jaggi2013revisiting}.

To measure the progress of the gradient descent step in Algorithm \ref{CG-G method}, we introduce below the smoothness constant of $f$ with respect to the subspace $T$:
\begin{definition}
	The smoothness parameter of $f$ with respect to the subspace $T$ and level set $X^0$ is defined as the Lipschitz constant of $\nabla f$ along subspace $T$:
	\begin{equation}\label{def: L}
		L_{f,x^0}^T := \sup_{\substack{x,y\in T\oplus S\\ x \neq y}} \lt\{ \frac{\|\cP_T\na f(x) - \cP_T\na f(y)\|_2}{\|x-y\|_2} ~\Big|~   x-y \in T,\ x,y\in X^0\rt\} \ .
	\end{equation}
\end{definition}

Now we are ready to present the sublinear rate of uFW:

{\color{black}
	\begin{theorem}\label{conv-rate-Cf}
		Consider Algorithm \ref{CG-G method} with initial solution $x^0$, gradient descent step-size $\eta \le 1/L_{f,x^0}^T$, and Frank-Wolfe step-size $\alpha_k$ following either the line-search rule~\eqref{line-search step-size rule} or the simple step-size rule~\eqref{simple step-size rule}.
		Then it holds for all $k\ge 1$ that
		$$
		f(x^{k})-f^{*}\le\frac{2C_{f,x^0}^S + 16  D_T^2 / \eta}{k+2} 
		$$
		where $f^*$ is the optimal objective value for problem \eqref{problem1}.
	\end{theorem}
	\begin{remark}
		We consider two special cases of Theorem~\ref{conv-rate-Cf}. When $T=\{0\}$, Algorithm \ref{CG-G method} recovers the traditional Frank-Wolfe method, and Theorem \ref{conv-rate-Cf} implies $f(x^{k})-f^{*}\le O({C_{f,x^0}^S}/{k})$, which recovers the standard convergence rate for the traditional Frank-Wolfe method. When $S=\{0\}$, Algorithm \ref{CG-G method} recovers the traditional gradient descent method, and Theorem \ref{conv-rate-Cf} implies $f(x^{k})-f^{*}\le O(L_{f,x^0}^T D_T^2/{k})$ if we take $\eta = 1/L_{f,x^0}^T$. This recovers the standard sublinear rate for gradient descent. 
	\end{remark}

		\subsubsection{Practical termination rules}\label{terminate-rules-FW-1}
		It is well known that the FW method (for problem \eqref{problem1} with $T=\{0\}$) naturally produces a primal-dual gap that can be used as a termination criterion (see, e.g. \cite{Freund-2016-analysis}). On the other hand, the gradient descent method (for problem \eqref{problem1} with $S=\{0\}$)
		can be terminated if the norm of the gradient is below a specified tolerance. In light of these results, we propose to set a termination criteria for Algorithm~\ref{CG-G method} if both the following quantities
		\begin{equation}\label{def: Gk and Hk}
			G_k := \la \na f(y^k) , \cP_T^\perp y^k - s^k \ra , \quad \text{and} ~ ~~~ H_k:=  \| \cP_T \na f(y^k) \|_2
		\end{equation}
		are small.
		Note that both $G_k$ and $H_k$ are always non-negative, and they can be computed as a by-product of Algorithm \ref{CG-G method}. The next proposition shows that 
		both $\min_{1\le k \le 2m-1} \{G_k\}$ and $\min_{1\le k \le 2m-1} \{H_k\} $ converge to $0$ with a rate $O(1/m)$. 
		
		\begin{proposition}\label{prop: stop-criterion-conv}
			Under the same assumption of Theorem \ref{conv-rate-Cf}, it holds for all $m\ge 1$ that
			\begin{equation}\label{ineq1: stop-criterion-conv}
				\min_{k=1,\ldots,2m-1} G_k \le \frac{1}{\log(4/3)} \frac{2C_{f,x^0}^S+8D_T^2/\eta}{m+1}\ ,
			\end{equation}
			and 
			\begin{equation}\label{ineq2: stop-criterion-conv}
				\min_{k=1,\ldots,2m-1} H_k \le \sqrt{\frac{8C_{f,x^0}^S\eta +32 D_T^2  }{\log(4/3)}} \left(\frac{\eta^{-1}+ L_{f,x^0}^T}{m+1}\right).
			\end{equation}
			
		\end{proposition}
		Finally, note that if some additional problem-specific parameters are known, 
		we can compute upper bounds for the primal optimality gap $f(y^k) - f^*$ based on $G_k$ and $H_k$. 
		\begin{proposition}\label{prop: primal-gap-with-GH}
			Under the same assumption of Theorem \ref{conv-rate-Cf}, we have
			\begin{equation}
				f(y^k) - f^* \le G_k + H_k D_T~~\forall k \geq 1.
			\end{equation}
			If in addition $f(\cdot)$ is $\mu$-strongly convex with respect to $\|\cdot\|_2$ (Definition \ref{strong-convex-def}), it holds 
			\begin{equation}
				f(y^k) - f^* \le G_k + H_k^2/(2\mu)~~\forall k \geq 1.
			\end{equation}
		\end{proposition}

	\subsection{Proofs of Theorem \ref{conv-rate-Cf} and Proposition \ref{prop: stop-criterion-conv}}
	We first present Lemmas~\ref{lemma-zero} and~\ref{claim1} before presenting the proofs of Theorem \ref{conv-rate-Cf} and Proposition \ref{prop: stop-criterion-conv}:
	\begin{lemma}\label{lemma-zero}
		For any $x,y\in  T\oplus S$ with $y-x \in T$ and $x,y\in X^0$, the following holds
		$$
		f(y) \le f(x) + \la \na f(x), y-x \ra + (L_{f,x^0}^T/2) 
		\|y-x\|_2^2 \ .
		$$
	\end{lemma}
	The proof of Lemma \ref{lemma-zero} is the same as that appearing in the proof of Lemma 1.2.3 in \cite{Nesterov-intro-book} except that we restrict the analysis to the case $ y-x \in T $. 
	
	\smallskip
	
	For notational convenience we let $z^{k+1}:=y^{k}+(2/(k+2))(s^k - \cP_T^{\perp} x^k)$ for $k\ge 0$.
	\begin{lemma}\label{claim1}
		Suppose we set $\eta \le 1/L_{f,x^0}^T$. 
		Then for both step-size rules \eqref{line-search step-size rule} and \eqref{simple step-size rule}, we have 
		$x^k\in X^0 $ for all $k\ge 0$ and
		$f(x^{k+1})\le f(z^{k+1})$ for all $k\ge 1$.
	\end{lemma}
	\textbf{Proof:}
	We first note that the conclusions hold true trivially for 
	the line-search step-size rule \eqref{line-search step-size rule}. For the simple step-size rule \eqref{simple step-size rule}, we first show that $f(x^k)\le f(x^0)$ for any $k\ge 0$. Suppose there exists a $k$ such that $f(x^k)> f(x^0)$, then without loss of generality, we assume $k\ge 1$ is the smallest index that satisfies $f(x^k)> f(x^0)$, thus it holds that $f(x^{k-1})\le f(x^0)$. From the simple step-size rule in \eqref{simple step-size rule}, if $f(z^k)\le f(x^0)$, it holds $x^k=z^k$, whereby $f(x^k)=f(z^k)\le f(x^0)$. Otherwise, $\al_{k-1}=0$, hence it holds $x^k=y^{k-1}$, whereby $f(x^k)= f(y^{k-1})\le f(x^{k-1})\le f(x^0)$. In either case, we have $f(x^k)\le f(x^0)$, which contradicts the assumption. Therefore, we have $f(x^k)\le f(x^0)$ for any $k\ge 0$. Consequently, if $f(z^{k+1})\le f(x^0)$, we have $x^{k+1}=z^{k+1}$ whereby $f(x^{k+1})\le f(z^{k+1})$. Otherwise we have $f(z^{k+1})>f(x^0)$, whereby $f(x^{k+1})\le f(x^0)<f(z^{k+1})$, therefore it always holds that $f(x^{k+1})\le f(z^{k+1})$. This completes the proof of Lemma \ref{claim1}.\qed
	
	\smallskip
	\noindent
	\textbf{Proof of Theorem \ref{conv-rate-Cf}.}
	First, note that
	\begin{equation}\label{combine2}
		\begin{aligned}
			f(y^k) & \mathop{\le}\limits^{(i)}
			f(x^k) + \la \na f(x^k) , y^k-x^k \ra +  \|y^k-x^k\|_2^2/(2\eta) \\
			&\mathop{=}\limits^{(ii)}
			f(x^k) 
			-{(\eta/2)}\|\cP_T \na f(x^k) \|_2^2 \le 
			f(x^k) \le f(x^0)
			\ ,
		\end{aligned}
	\end{equation}
	where $(i)$ uses Lemma \ref{lemma-zero} and the assumption $\eta \le 1/L_{f,x^0}^T$, and $(ii)$ is from the update rule $y^k = x^k - \eta \cP_T \na f(x^k)$.
	Let $x^*$ be an optimal solution of \eqref{problem1}.
	Recall that we have defined $z^{k+1}:=y^{k}+({2}/{(k+2)})(s^k - \cP_T^{\perp} x^k)$, and shown in Lemma \ref{claim1} that $f(x^{k+1}) \le f(z^{k+1}) $.
	As a result, we have
	\begin{equation}\label{ieq3}
		\begin{aligned}
			f(x^{k+1}) & \le  f(z^{k+1}) \\
			& \le 
			f(y^k) + \frac{2}{k+2} \la \na f(y^k) , s^k - \cP_T^\perp x^k \ra + \frac{2C_{f,x^0}^S}{(k+2)^2} \\
			&\le 
			f(y^k) +  \frac{2}{k+2} \la \na f(y^k) , \cP_T^\perp x^* - \cP_T^\perp y^k \ra + \frac{2C_{f,x^0}^S}{(k+2)^2} \ ,
		\end{aligned}
	\end{equation}
	where the
	second inequality is from the definition of $C_{f,x^0}^S$ with $s=s^k,x=y^k$, $y=z^{k+1}$ and $\al = {2}/{(k+2)}$, and the third inequality is by the definition of $s^k$ 
	and $\cP_{T}^\perp x^k = \cP_{T}^\perp y^k$.
	Since
	\begin{equation}\label{combine3}
		\la \na f(y^k) , \cP_T^\perp x^* - \cP_T^\perp y^k \ra = 
		\la  \na f(y^k) ,  x^* -  y^k \ra + 
		\la  \cP_T \na f(y^k) , y^k - x^*   \ra \ ,
	\end{equation}
	we have
	\begin{equation}\label{ieq4}
		\begin{aligned}
			f(x^{k+1}) & \mathop{\le}\limits^{(i)} 
			\frac{2}{k+2} f(y^k) + \frac{k}{k+2} f(x^k) - \frac{\eta}{2} \frac{k}{k+2} \|\cP_T \na f(x^k) \|_2^2 + \frac{2C_{f,x^0}^S}{(k+2)^2} \\
			&~~~ +\frac{2}{k+2} \la  \na f(y^k) ,  x^* -  y^k \ra + 
			\frac{2}{k+2} \la  \cP_T \na  f(y^k) , y^k - x^*   \ra  \\
			& \mathop{\le}\limits^{(ii)} \frac{2}{k+2} f^* + \frac{k}{k+2} f(x^k) + \frac{2C_{f,x^0}^S}{(k+2)^2} \\
			& ~~~ + 
			\frac{2}{k+2} \la  \cP_T \na f(y^k) , y^k - x^*   \ra -  \frac{\eta}{2} \frac{k}{k+2} \|\cP_T \na f(x^k) \|_2^2 \ ,
		\end{aligned}
	\end{equation}
	where $(i)$ utilizes \eqref{combine2}, \eqref{ieq3} and \eqref{combine3}, and $(ii)$ 
	is because $f(y^k)  + \la \na f(y^k), x^* - y^k \ra \le f^*$ due to the convexity of $f$. Note that
	\begin{equation}\label{ieq5}
		\begin{aligned}
			& \la \cP_T \na f(y^k), y^k-x^* \ra \\
			&= 
			\la  \cP_T \na f(y^k) - \cP_T \na f(x^k) , y^k-x^* \ra + \la \cP_T \na f(x^k), y^k-x^* \ra \\
			&\mathop{\le}\limits^{(i)}
			L_{f,x^0}^T \|y^k-x^k\|_2 D_T + \|\cP_T \na f(x^k)\|_2 D_T \\
			& \mathop{=}\limits^{(ii)}
			\lt( L_{f,x^0}^T \eta  +1  \rt) \|\cP_T \na f(x^k)\|_2 D_T  \mathop{\le}\limits^{(iii)}
			2\|\cP_T \na f(x^k)\|_2 D_T \ ,
		\end{aligned}
	\end{equation}
	where in the above display, $(i)$ follows from the definitions of $D_T$ and $L_{f,x^0}^T$ in \eqref{def: DT} and \eqref{def: L} and the fact $\{y^k,x^*\}\subset X^0$, and $(ii)$ is due to the
	update rule in step (1) of Algorithm \ref{CG-G method}, and $(iii)$ utilizes the
	assumption $\eta \le 1/ L_{f,x^0}^T$. 
	By combining \eqref{ieq4} and \eqref{ieq5}, we arrive at
	\begin{equation}\label{ieq6}
		\begin{aligned}
			f(x^{k+1}) & \le 
			\frac{2}{k+2} f^* + \frac{k}{k+2} f(x^k) + \frac{2C_{f,x^0}^S}{(k+2)^2} \\
			& ~~~ + 
			\frac{4D_T}{k+2} \|\cP_T \na f(x^k)\|_2
			-  \frac{\eta}{2} \frac{k}{k+2} \|\cP_T \na f(x^k) \|_2^2
		\end{aligned} 
	\end{equation}
	for all $k\ge0$.
	For $k\ge 2$, using \eqref{ieq6} and the following inequality 
	$$
	\frac{8 D_T^2}{k(k+2)\eta}+ \frac{\eta}{2} \frac{k}{k+2} \|\cP_T \na f(x^k) \|_2^2\ge \frac{4D_T}{k+2} \|\cP_T \na f(x^k)\|_2 \ ,
	$$
	(which follows from $a^2+b^2\geq 2ab$ for scalars $a,b$), we have
	\begin{eqnarray}
		f(x^{k+1}) & \mathop{\le} & \frac{2}{k+2} f^* + \frac{k}{k+2} f(x^k) + \frac{2C_{f,x^0}^S}{(k+2)^2} + \frac{8D_T^2}{k(k+2) \eta} \nonumber\\
		&\mathop{\le}\limits^{(i)}& \frac{2}{k+2} f^* + \frac{k}{k+2} f(x^k) + \frac{2C_{f,x^0}^S+16  D_T^2/\eta}{(k+2)^2} \ .  \nonumber
	\end{eqnarray}
	Note that 
	inequality $(i)$ above makes use of  $k\ge2$. 
	
	Let $h_k:= f(x^k) - f^*$ and $C:= 2C_{f,x^0}^S+16  D_T^2/\eta$, then we have 
	\begin{equation}\label{eq:induction}
		h_{k+1} \le {k h_{k}}/{(k+2)} + {C}/{(k+2)^2}, \quad \forall k\ge 2 \ .
	\end{equation}
	To complete the proof, we will show that $h_k\le {C}/{(k+2)}$
	by induction. First, it follows from \eqref{ieq6} by choosing $k=0$ that
	$$
	h_1 \le (1/2) C_{f,x^0}^S + 2D_T \|\cP_T \na f(x^0)\|_2\ .
	$$
	Let $w^0$ be a minimizer of $f(x)$ on the affine subspace $x^0 + T$ (the existence of a minimizer is guaranteed by the assumption that $D_T<\infty$), then it holds that $w^0\in X^0$, $w^0-x^0 \in T$ and $\cP_T \na f(w^0) = 0$. Thus we have
	\begin{equation}\label{eq--1}
		\begin{aligned}
			h_1 &\le (1/2) C_{f,x^0}^S +2D_T \|\cP_T \na f(x^0)- \cP_T \na f(w^0)\|_2 
			\\
			& \le (1/2) C_{f,x^0}^S + 2 L_{f,x^0}^T D_T^2 \le C/3 \ .
		\end{aligned}
	\end{equation}
	Again, using \eqref{ieq6} with $k=1$, we have
	\begin{equation}\label{eq--2}
		\begin{aligned}
			h_2 &\le (1/3) h_1 + (2/9)  C_{f,x^0}^S + (4/3) D_T \| \cP_T \na f(x^1) \|_2 \\
			&\le (1/3) \lt(  (1/2) C_{f,x^0}^S + 2 L_{f,x^0}^T D_T^2\rt) + (2/9) C_{f,x^0}^S +
			(4/3)L_{f,x^0}^T D_T^2 \\
			&\le C/4.
		\end{aligned}
	\end{equation}
	Inequalities \eqref{eq--1} and \eqref{eq--2} imply that
	$h_k\le C/(k+2)$ holds for $k=1, 2$. Now suppose $h_k\le C/(k+2)$ holds for some $k\ge 2$, then it follows from \eqref{eq:induction} that $h_{k+1} \le {C}/{(k+3)}$. 
	Therefore, $h_k \le C/(k+2)$ $\forall k \ge1$ by induction---this finishes the proof by substituting the value of $C$.\qed
	
}

	\smallskip
	\noindent
	\textbf{Proof of Proposition \ref{prop: stop-criterion-conv}.}
	Let $\tau_k:= 2/(k+2)$ and $\wtd H_k := \| \cP_{T} \na f(x^k) \|_2 $ (note that $\wtd H_k$ is different from $H_k$). From \eqref{combine2}, \eqref{ieq3} and noting that $\cP_T x^k = \cP_T y^k  $, we have 
	\begin{equation}\label{new-proof-rev-ineq-1}
		f(x^{k+1}) \le f(x^k) - \tau_k G_k + ({\tau_k^2}/{2}) C_{f,x^0}^S - (\eta/2)\wtd H_k^2 \ . 
	\end{equation}
	Summing inequality~\eqref{new-proof-rev-ineq-1} from $k=m$ to $k=2m+1$ and rearranging terms we have 
	\begin{equation}\label{ineq--new1}
		\sum_{k=m}^{2m-1} (\tau_k G_k + (\eta/2)\wtd H_k^2) \le 
		f(x^m) - f(x^{2m-1}) + \sum_{k=m}^{2m-1} \frac{\tau_k^2}{2} C_{f,x^0}^S \ .
	\end{equation}
	Note that 
	\begin{equation}\label{ineq--new2}
		\sum_{k=m}^{2m-1} \frac{\tau_k^2}{2} =  \sum_{k=m}^{2m-1} \frac{2}{(k+2)^2} \le 2\sum_{k=m}^{2m-1} \Big( \frac{1}{k+1} - \frac{1}{k+2} \Big) \le \frac{2}{m+1} \ .
	\end{equation}
	Combining \eqref{ineq--new1}, \eqref{ineq--new2}, and noting that $f(x^m) - f(x^{2m-1}) \le  f(x^m) - f^*$, we have 
	\begin{equation}\label{ineq--new3}
		\sum_{k=m}^{2m-1} (\tau_k G_k + (\eta/2)\wtd H_k^2) \le f(x^m) - f^* + \frac{2C_{f,x^0}^S}{m+1} \ .
	\end{equation} 
	By \eqref{ineq--new3} and the conclusion of Theorem \ref{conv-rate-Cf} we have 
	\begin{equation}\label{ineq--new4}
		\sum_{k=m}^{2m-1} (\tau_k G_k + (\eta/2)\wtd H_k^2) \le \frac{2C_{f,x^0}^S + 16  D_T^2 / \eta}{m+2} + \frac{2C_{f,x^0}^S}{m+1} \le \frac{4C_{f,x^0}^S + 16  D_T^2 / \eta}{m+1} \ .
	\end{equation} 
	Note that 
	\begin{equation}\label{ineq--new4.5}
		\sum_{k=m}^{2m-1} \tau_k = \sum_{k=m}^{2m-1} \frac{2}{k+2} \ge \int_{m}^{2m} \frac{2}{t+2} dt = 2 \log\frac{2m+2}{m+2} \ge 2 \log(4/3) \ .
	\end{equation}
	So by \eqref{ineq--new4} and \eqref{ineq--new4.5}, there exists some $m\le j \le 2m-1$ satisfying
	\begin{equation}\label{ineq--new5}
		G_j + \eta \wtd H_j^2 /(2\tau_j)  \le \Big( \sum_{k=m}^{2m-1} \tau_k  \Big)^{-1} \sum_{k=m}^{2m-1} (\tau_k G_k + (\eta/2)\wtd H_k^2)
		\le \frac{2C_{f,x^0}^S + 8  D_T^2 / \eta}{(m+1)\log(4/3)} \ .
	\end{equation}
	The first 
	inequality~\eqref{ineq1: stop-criterion-conv} follows immediately from \eqref{ineq--new5}. To prove the second inequality of this proposition, note that 
	\begin{equation}\label{ineq--new7}
		\begin{aligned}
			H_j - \wtd H_j &\le \| \cP_T \na f(x^j) - \cP_T \na f(y^j) \|_2 \le L_{f,x^0}^T \| x^j - y^j \|_2 \\
			&\le
			L_{f,x^0}^T \eta \| \cP_T \na f(x^j) \|_2 = L_{f,x^0}^T \eta \wtd H_j \ .
		\end{aligned}
	\end{equation}
	Combining \eqref{ineq--new5} and \eqref{ineq--new7} we have 
	\begin{equation}\label{ineq--new8}
		H_j \le (1+\eta L_{f,x^0}^T) \wtd H_j \le (1+\eta L_{f,x^0}^T) \Big( \frac{2\tau_j}{\eta} \frac{2C_{f,x^0}^S + 8  D_T^2 / \eta}{(m+1)\log(4/3)}  \Big)^{1/2} \ .
	\end{equation}
	Recall that $\tau_j = 2/(j+2) \le 2/(m+1)$, so by \eqref{ineq--new8} we arrive at~\eqref{ineq2: stop-criterion-conv}, which completes the proof. \qed

	\textbf{Proof of Proposition \ref{prop: primal-gap-with-GH}.}
	Let $x^*$ be an optimal solution of \eqref{problem1}. 
	By the convexity of $f(\cdot)$, we have 
	\begin{equation}
		\begin{aligned}
			f(y^k) - f^* &\le \la \na f(y^k), y^k - x^* \ra \\
			&= \la \na f(y^k), \cP_T^\perp y^k - \cP_T^\perp x^* \ra + 
			\la \cP_T \na f(y^k), \cP_T (y^k -  x^*)\ra \\
			&\le 
			\la \na f(y^k), \cP_T^\perp y^k -  s^k \ra + \| \cP_T \na f(y^k)\|_2 \| \cP_T (y^k -  x^*)\|_2 \\
			&\le G_k + H_k D_T \ ,
		\end{aligned}
	\end{equation}
	where the second inequality is by the definition of $s^k$ and the third inequality is by the definition of $D_T$. 
	
	If in addition $f(\cdot)$ is $\mu$-strongly convex with respect to $\|\cdot\|_2$, then 
	\begin{equation}
		\begin{aligned}
			f(y^k) - f^* &\le \la \na f(y^k), y^k - x^* \ra - \frac{\mu}{2} \| y^k - x^* \|_2^2 \\
			&= \la \na f(y^k), \cP_T^\perp y^k - \cP_T^\perp x^* \ra + 
			\la \cP_T \na f(y^k),  y^k -  x^*\ra - \frac{\mu}{2} \| y^k - x^* \|_2^2 \\
			&\le 
			\la \na f(y^k), \cP_T^\perp y^k -  s^k \ra +  \| \cP_T \na f(y^k)\|^2_2/(2\mu)\\
			&\le G_k + H_k^2/(2\mu)
		\end{aligned}
		\nonumber
	\end{equation}
	where the second inequality uses Cauchy-Schwarz inequality and the definition of $s^k$.

\subsection{Linear Rate for uAFW}\label{subsection: analysis of uAFW}

In this section, we establish the global linear convergence of uAFW (Algorithm \ref{CG-G with away steps}).

The standard linear convergence result for away-step Frank-Wolfe method relies on a suitable curvature constant of $f$ with respect to an extended constraint set~\cite{lacoste2015global} (equivalently, it is the relative smoothness constant to the constraint set presented in~\cite{gutman2019condition}). Here we generalize this notion to the setting of~\eqref{problem1}.
\begin{definition}
	We define $\bar C_{f,x^0}^S$, the curvature constant of $f$ with respect to an extended constraint set along the set $S$ as follows:
	\begin{equation}
		\begin{aligned}
			\bar C_{f,x^0}^S:=
			\sup_{\substack{s\in S, \  x,y\in \R^n \\\al\in [-1,1] \backslash \{0\}}} 
			\big\{ & ({2}/{\al^2})\lt(  f(y) - f(x) - \la  \na f(x), y-x \ra  \rt)    ~\big| \\
			& ~~~~	\cP_T^{\perp} x \in S,\ f(x)\le f(x^0),\ 
			y = x+ \al (s - \cP_T^\perp x)
			\big\} \ .
		\end{aligned}
		\nonumber
	\end{equation}
\end{definition}

Note that $\bar C_{f,x^0}^S$ is different from
$ C_{f,x^0}^S$ defined in \eqref{def: Cf}.
In the definition of $\bar C_{f,x^0}^S$, $\al$ is allowed to take values in $[-1,1] \backslash \{0\}$, while in the definition of $ C_{f,x^0}^S $ it takes values in $(0,1]$. In other words, $\bar C_{f,x^0}^S$ quantifies the curvature of $f$ with respect to an extended set $\bar S:= S+(S-S)$ (i.e., the Minkowski sum of $S$ and $S-S$). 
In the special case when $f$ is $L$-smooth over $\bar S\oplus T$, if we define  $\bar D:= \diam (\bar S)$, then we have $\bar C_{f,x^0}^S \le L \bar D^2 $.
Intuitively, $\bar C_{f,x^0}^S$ 
should
appear in the computational guarantees for uAFW
because the algorithm not only moves towards the vertices but also away from vertices, and $\bar C_{f,x^0}^S$ provides an upper bound on the curvature information of $f$ when moving in these directions.

To obtain a linear rate for uAFW we also need to assume strong convexity of the objective function.

\begin{definition}\label{strong-convex-def}
	Let X be a convex set in $\R^n$. For $\mu>0$, we say a function f is $\mu$-strongly convex on $X$ with respect to a norm $\|\cdot\|$
	if for all $x,y\in X$, we have
	$$
	f(y) - f(x) - \la \na f(x),y-x \ra \ge (\mu/2) \|y-x\|^2 \ .
	$$
\end{definition}

Finally, following \cite{pena2019polytope}, we define the facial distance $\psi(S)$ of the polyhedral set $S$ as follows:
\begin{equation}\label{facial-distance-defn}
	\psi(S) := \min_{\mbox{$\scriptsize{
				\begin{array}{c}
					F \in \text{faces}(S), \\
					\emptyset \neq F \neq S
				\end{array}
			} $}}
	\text{dist}(F , \text{conv}(V(S) \backslash F)) \ ,
\end{equation}
where $V(S)$ denotes the set of all vertices of $S$, $\text{faces}(S)$ the set of all faces of $S$; and $V(S) \backslash F$ means $V(S) \cap F^c$.
Note $\psi(S)$
appears in the computational guarantees of the classical away-step Frank-Wolfe algorithm in the analysis of~\cite{pena2019polytope}. It is strictly positive for a polyhedral $S$, but may be zero for a general convex set $S$.

\smallskip

We make the following assumptions to establish the linear convergence rate.
\begin{assumption}\label{awaystep-ass}
	(1) $S$ is a bounded polyhedron.
	
	(2)	$X^0$ is bounded. \quad
	
	(3) $f$ is $L$-smooth over the extended constraint $\bar S \oplus T$ with $\bar S:= S+(S-S) = \{x+(y-z)~|~x,y,z \in S\}$.

	(4) $f$ is $\mu$-strongly convex on $S\oplus T$ with respect to $\|\cdot\|_2$.

\end{assumption}

As discussed before, condition (3) of Assumption \ref{awaystep-ass} guarantees that the parameter $\bar C_{f,x^0}^S $ is bounded above. 
The next theorem presents the linear convergence of 
Algorithm \ref{CG-G with away steps}.

\begin{theorem}\label{CGG-linear-conv}
	Suppose that $f$ is convex and continuously differentiable, and satisfies Assumption \ref{awaystep-ass}. 
	Let $\{(x^k,y^k)\}_{k\ge0}$ be the sequence generated by Algorithm \ref{CG-G with away steps} by setting $\eta \le 1/L$; and recall that $f^*$ is the optimal objective value. Let $\sigma~=~\mu\psi(S)^2 / 4$,
	then we have
	\begin{equation}\label{CGG-linear-conv-eq}
		f(y^k) - f^* \le \left(1- \min\lt\{1/2, {\sigma}/{ \bar C_{f,x^0}^S}\rt\} \frac{\mu\eta}{4} \right)^{k/2} (f(y^0) - f^*) \ .
	\end{equation}
\end{theorem}

\begin{remark}
	If we set the stepsize $\eta = 1/L$ (recall, $f$ is $L$-smooth over $\bar S \oplus T $),
	the linear convergence rate in \eqref{CGG-linear-conv-eq} depends on two ratios: $\sigma/ \bar C_{f, x^0}^S$ and $ \mu/L$. The first ratio is in accordance with the linear rate parameter for the classical away-step Frank-Wolfe algorithm, while the second one is the rate for classical gradient descent. 
\end{remark}

	\subsubsection{Practical termination rules}
	Similar to Section~\ref{terminate-rules-FW-1}, we can use quantities 
	$G_k$ and $H_k$ defined in \eqref{def: Gk and Hk} (with $y^k$ being the iterations produced by uAFW) to derive a practical termination criteria for uAFW. 
	The following proposition which is a variant of Theorem 2 in \cite{lacoste2015global}
	presents the rate at which $G_k$ and $H_k$ converge to zero. 
	\begin{proposition}\label{prop: stop-criterion-conv2}
		Under the same assumption of Theorem \ref{CGG-linear-conv}, it holds:
		
		(1) $H_k \le (1+ \eta L^T_{f,x^0}) \sqrt{(2/\eta)(f(x^k) - f^*)}$.
		
		(2) 
		Among the first $k$ iterations of Algorithm \ref{CG-G with away steps}, there are at least $\lfloor k/2 \rfloor$ iterations satisfying: 
		\begin{equation}\label{ineq: stop-criterion-conv2}
			G_k \le \min \Big\{ 2(f(y^k) - f^*), \sqrt{2 \bar C_{f,x^0}^S (f(y^k)-f^*) } \Big\}. 
		\end{equation}
	\end{proposition}
	The proof of Proposition \ref{prop: stop-criterion-conv2} is relegated to Appendix \ref{appendix sec: proof of prop-stop-criterion-conv2}. In light of Theorem \ref{CGG-linear-conv}, Proposition \ref{prop: stop-criterion-conv2} and noting that $f(x^{k+1}) \le f(y^k)$, it can be seen that $G_k$ and $H_k$ converge to $0$ with a linear rate. Hence we can terminate Algorithm \ref{CG-G with away steps} when $G_k$ and $H_k$ are smaller than a tolerance value. \\
	The conclusions of Proposition \ref{prop: primal-gap-with-GH} also hold true for Algorithm \ref{CG-G with away steps}. These results can be used to compute upper bounds on $f(y^k) - f^*$ if $D_T$ or $\mu$ is known.

\subsection{Proof of Theorem \ref{CGG-linear-conv}}

We divide the proof of Theorem \ref{CGG-linear-conv} into the following three steps.

\smallskip

\noindent \textbf{Step 1. A reduction to the axis-aligned case:} Notice that both Frank-Wolfe and gradient descent are invariant under orthogonal transformations; and so is uAFW. 
 That is, given an orthogonal transformation $Q$, 
	suppose $\{x^k,y^k\}_{k=0}^{\infty}$ are the  iterations generated by uAFW for the objective function $f(\cdot)$, constraint set $T\oplus S$ 
	and initialization $x^0$. Suppose 
	$\{\td x^k, \td y^k\}_{k=0}^{\infty}$ are the iterations generated by uAFW for the objective function $g(\cdot) := f(Q^{-1}\cdot)$, constraint set $Q(T) \oplus Q(S)$ (where $Q(T):= \{Qx~|~ x\in T\}$ and $Q(S)$ is defined similarly) and initialization $\td x^0 = Qx^0$. Then it holds $\td x^k = Qx^k$ and $\td y^k = Q y^k$ for all $k\ge 0$. Assumption \ref{awaystep-ass} also holds for the function and constraints obtained after the orthogonal transformation.
Thus without loss of generality, we can assume that $\{e_{n-r+1},e_{n-r+2},\ldots, e_n\}$ is a basis of the $r$-dimensional subspace $T$.
Then we can convert the original problem \eqref{problem1} to the following form:
\begin{equation}\label{problem1-eq3}
	\min_{u,w} \ \bar f(u,w) ~~~ {\rm s.t.}~~~u \in S\subset \R^{n-r} \ , w \in \R^{r} \ ,
\end{equation}
where $S\subset \R^{n-r}$ is a bounded polyhedral set (with a slight abuse, we are reusing the notation $S$). We refer to~\eqref{problem1-eq3} as the \emph{axis-aligned form}. Although the transformed problem~\eqref{problem1-eq3} is equivalent to the original problem~\eqref{problem1}, in practice it can be computationally expensive to perform such transformation. We
use formulation~\eqref{problem1-eq3} 
to help facilitate our proof.

Let $\na_u \bar f(u,w)$ and $\na_w \bar f(u,w)$  denote the sub-vectors of $\na \bar f(u,w)$ corresponding to the derivatives of $u$ and $w$ respectively. 
Then we can rewrite uAFW (Algorithm \ref{CG-G with away steps}) as Algorithm~\ref{Away-step CG-G (normal form)}.

\begin{algorithm}
	\caption{uAFW for axis-aligned form problem \eqref{problem1-eq3}}
	\label{Away-step CG-G (normal form)}
	\begin{algorithmic}
		\STATE Starting from $(u^{0},w^{0})$ such that $u^0$ is a vertex of $ S $.
		\STATE For $k=0,1,2, \dots$:
		\STATE (1) $w^{k+1} = w^k - \eta \na_w \bar f(u^k,w^k) $.
		\STATE (2) 
		$s^{k}:= \arg\min_{s\in S}\la\na_u \bar f(u^k,w^{k+1}),s\ra$,
		$v^k:= \argmax_{v \in V(u^k)} \la \na_u \bar f(u^k,w^{k+1}) ,v \ra $.
		\STATE \qquad \textbf{if} 
		$ \la \na_u \bar f(u^k,w^{k+1}) , s^k - u^k \ra < 
		\la \na_u \bar f(u^k,w^{k+1}), u^k - v^k \ra $: (Frank-Wolfe step)
		\STATE \qquad \quad $d^k:= s^k - u^k$; $\al_{\max} = 1$.
		\STATE \qquad \textbf{else}: (away step)
		\STATE \qquad \quad $d^k:= u^k - v^k$; $\al_{\max} = {\lam_{v^k}(u^k)  }/({1- \lam_{v^k}(u^k)})$.
		\STATE (3)  $u^{k+1} = u^{k}+\alpha_{k}d^k$ ; 
		where $\alpha_{k} \in \argmin\nolimits_{\al\in[0,\al_{\max}]} f(u^k + \al d^k, w^{k+1} )$.
		\STATE (4) Update $V(u^{k+1})$ and $\lam(u^{k+1})$.
	\end{algorithmic}
\end{algorithm}

Note that in Algorithm \ref{Away-step CG-G (normal form)}, we use $V(u^{k+1})$ and $\lam(u^{k+1})$ in place of the corresponding $V(x^{k+1})$ and $\lam(x^{k+1})$ in Algorithm \ref{CG-G with away steps} because their values only depend on $u^{k+1}$. 
More specifically, we update $V(u^{k+1})$ and $\lam(u^{k+1})$ in the following way: For a Frank-Wolfe step, let $V(u^{k+1}) = \{s^k\}$ if $\al_k=1$; otherwise $V(u^{k+1}) = V(u^k) \cup \{s^k\}$. Meanwhile, we set $ \lam_{s^k}(u^{k+1}) := (1-\al_k) \lam_{s^k} (u^k) + \al_k $ and $\lam_v(u^{k+1}) = (1-\al_k) \lam_v (u^{k+1}) $ for $v\in V(u^k) \backslash \{s^k\}$. For an away step, we update $V(u^{k+1}) = V(u^{k}) \backslash \{v^k\} $ if $\al_k = \al_{\max}$; otherwise we set $V(u^{k+1}) = V(u^k)$. We set $ \lam_{v^k} (u^{k+1}) := (1+ \al_k) \lam_{v^k} (u^k) - \al_k $ and $ \lam_{v}(u^{k+1}) := (1+\al_k) \lam_v^{k} $ for $v\in V(u^k) \backslash \{v^k\}$.

\smallskip\smallskip

Below we present the analysis of  Algorithm \ref{Away-step CG-G (normal form)} for 
Problem~\eqref{problem1-eq3}.

\smallskip

\noindent \textbf{Step 2. Reduction in optimality gap via a Frank-Wolfe step:} We establish a reduction in optimality gap via the Frank-Wolfe step. This is an adaption of Proposition 9 in \cite{gutman2019condition} to the case of a function $\bar f(\cdot, w)$ with a dynamic $w$. To this end, we define $F: \R^{r} \rightarrow \R $ as $F(w) = \rminf_{u \in S} \bar f(u,w)$.

\begin{proposition}\label{FW-contraction}
	Consider Algorithm \ref{Away-step CG-G (normal form)} with initial solution $(u^0, w^0)$.  Define $\gamma := \mu \psi(S)^2 / (4\bar C_{f,x^0}^S) $ and $\rho:= \min \lt\{ 1/2, \ga\rt\}$.
	Then among iterations $1,\ldots,k$, there are at least $k/2$ of them satisfying
	$$
	\bar f(u^{k+1}, w^{k+1}) -  f^* \le (1-\rho) \lt( \bar f(u^k, w^{k+1}) -  f^* \rt) + \rho ( F(w^{k+1}) -  f^*) \ .
	$$
\end{proposition}

The proof of Proposition \ref{FW-contraction} is inspired by~\cite{gutman2019condition}---it turns out to be a bit technical and we relegate the proof to Appendix \ref{sec:appendix-prop}.

\smallskip

\noindent
\textbf{Step 3. Piecing together the Frank-Wolfe and gradient descent steps:} Here we establish the linear convergence rate of uAFW by combining the contraction of the gradient descent step along the unbounded space $T$, and the contraction of the Frank-Wolfe step along the bounded set $S$ (presented by Proposition \ref{FW-contraction}).

We first present a few technical lemmas that will be used in the proof. 
Recall that we defined $F: \R^{r} \rightarrow \R $ as $F(w) = \rminf_{u \in S} \bar f(u,w)$. Given $w\in \R^r$, we denote $ u_w^* \in \argmin_{u\in S} \bar f(u,w)$. In addition, let $w^*\in \argmin_{w\in \R^r} F(w)$ and $F^*=f^*= \min_{w\in \R^r} F(w)$.

\begin{lemma}\label{lemma:diff-under-inf} 
	Under Assumption \ref{awaystep-ass}, we have
	
	(1) $F$ is convex and differentiable over $\R^r$, with gradient $\na F(w) = \na_w \bar f(u^*_{w},w)$ .

	(2) $F$ is $\mu$-strongly convex on $\R^r$ with respect to the $\ell_{2}$-norm $ \| \cdot\|_2\  $.
	
	(3) For any $ w\in \R^r $, it holds 
	$F(w) - F^* \le \|\na F(w) \|_2^2/(2\mu) \ .$

\end{lemma}

\begin{proof}
	\underline{Part (1)} The convexity of $F$ can be found in a standard reference on convex analysis, see e.g. \cite{rockafellar1970convex}. The differentiability property of $F$ is known as Danskin's theorem (for example, see~\cite{bonnans1998optimization} for a proof). 
	
	\smallskip
	
	\noindent \underline{Part (2)} 
	Note that a function $g: \R^m \rightarrow \R$ is $\mu$-strongly convex with respect to $\|\cdot\|_2$ if and only if 
	\begin{equation}\label{strong-convex-equivalence}
		\la \na g(x) - \na g(y) , x-y \ra \ge \mu \|x-y\|_2^2 
	\end{equation}
	for all $x,y\in \R^m$ (See e.g. Theorem 2.1.9 of 
	\cite{Nesterov-intro-book}). 
	For any $w^1,w^2\in \R^r$ we have
	\begin{eqnarray}
		&& \la  \na F(w^2) -  \na F(w^1) , w^2 - w^1 \ra \nonumber\\
		&\mathop{=}\limits^{(i)}& 
		\lt\la \na_w \bar f(u^*_{w^2},w^2) - \na_w \bar f(u^*_{w^1},w^1), w^2 - w^1 \rt\ra \nonumber\\
		&=&
		\lt\la   \na \bar f(u^*_{w^2},w^2) - \na \bar f(u^*_{w^1},w^1),  (u^*_{w^2},w^2) - (u^*_{w^1},w^1)  \rt\ra \nonumber\\
		&&
		-
		\lt\la   \na_u \bar f (u^*_{w^2},w^2) -  \na_u \bar f (u^*_{w^1},w^1) ,  u^*_{w^2} - u^*_{w^1} \rt\ra \nonumber \\
		&\mathop{\ge}\limits^{(ii)}&
		\mu \lt\|  (u^*_{w^2},w^2) - (u^*_{w^1},w^1)  \rt\|_2^2
		-
		\lt\la   \na_u \bar f (u^*_{w^2},w^2) -  \na_u \bar f (u^*_{w^1},w^1) ,  u^*_{w^2} - u^*_{w^1} \rt\ra \nonumber\\
		&\ge&
		\mu\| w^2 - w^1 \|_2^2 -
		\lt\la   \na_u \bar f (u^*_{w^2},w^2) -  \na_u \bar f (u^*_{w^1},w^1) ,  u^*_{w^2} - u^*_{w^1} \rt\ra \nonumber \ .
	\end{eqnarray}
	where inequality $(i)$ follows from the conclusion in part (1), and inequality $(ii)$ follows from the strong-convexity of $\bar f$ and the equivalent condition \eqref{strong-convex-equivalence}.
	Meanwhile, it follows from $u^*_{w^1} \in \argmin_{u\in S} \{\bar f(u,w^1)\}$ and $ u^*_{w^2} \in \argmin_{u\in S} \{\bar f(u,w^2)\} $ that
	$$
	\lt\la    \na_u \bar f(u^*_{w^1}, w^1 ), u^*_{w^2} - u^*_{w^1}  \rt\ra \ge 0, \quad
	\lt\la    \na_u \bar f(u^*_{w^2}, w^2 ), u^*_{w^1} - u^*_{w^2}  \rt\ra \ge 0 \ .
	$$
	Therefore, we have 
	$
	\la  \na F(w^2) -  \na F(w^1) , w^2 - w^1 \ra \ge
	\mu \| w^2 - w^1 \|_2^2 \ ,
	$
	which shows $F$ is $\mu$-strongly convex with respect to $\|\cdot\|_2$. 
	
	\smallskip
	
	\noindent	\underline{Part (3)} Since $F$ is $\mu$-strongly convex over $\R^r$, this final claim follows from Theorem 2.1.10 of \cite{Nesterov-intro-book}. \end{proof}

\begin{lemma}\label{lemma:smooth-inequality}
	[Theorem 2.1.5 of 
	\cite{Nesterov-intro-book}]
	For any $u^1, u^2 \in S$ and $w^1, w^2\in \R^r$, we have
	\begin{equation}
		\begin{aligned}
			\bar f(u^2, w^2) \ge & ~ \bar f(u^1, w^1) + \lt\la \na \bar f(u^1,w^1), (u^2-u^1, w^2-w^1) \rt\ra \\
			&  + \frac{1}{2L} \| \na \bar f(u^2, w^2) - \na \bar f(u^1, w^1) \|_2^2 \ .
		\end{aligned}
		\nonumber
	\end{equation}
\end{lemma}

The next proposition combines the contraction of objective values in the gradient descent and Frank-Wolfe steps:
	\begin{proposition}\label{prop:combining}
		Suppose in iteration $k$ ($k\ge 1$)
		the following condition holds:
		\begin{equation}\label{contraction-condition}
			\bar f(u^{k+1}, w^{k+1}) - f^* \le (1-\rho) \lt( \bar f(u^k, w^{k+1}) - f^* \rt) + \rho (F(w^{k+1}) - f^*). 
		\end{equation}
		for some $\rho\in (0,1)$. 
		Then 
		\begin{eqnarray}\label{contraction-result}
			\bar f(u^{k+1},w^{k+2})  - f^* \le (1-\frac{\mu\eta}{4} \rho) \lt( \bar f(u^{k},w^{k+1})  - f^* \rt) \ .
		\end{eqnarray}
	\end{proposition}
\begin{proof}
	Let $\tau := {\mu\eta}/{4}$. Then $0< \tau \le1/4$. We split the analysis into two cases. 
	
	\smallskip
	
	\noindent
	{\bf (Case $i$)} Here we consider the case when 
	\begin{equation}\label{assumption-a}
		F(w^{k+1}) - f^* \le (1-\tau) \lt[ 	\bar f(u^k,w^{k+1})  - f^* \rt] \ .
	\end{equation}
	Then we have
	\begin{eqnarray*}
		\bar f(u^{k+1},w^{k+2})  - f^* &\le&
		\bar f(u^{k+1},w^{k+1})  - f^* \\
		&\le&
		(1-\rho) \lt( \bar f(u^k, w^{k+1}) - f^* \rt) + \rho (F(w^{k+1}) - f^*) \\
		&\le&
		(1-\rho) \lt[ \bar f(u^k, w^{k+1}) - f^* \rt] +\rho(1-\tau) \lt[ 	\bar f(u^k,w^{k+1})  - f^* \rt] \\
		&=&
		(1-\rho \tau) \lt[ \bar 	f(u^k,w^{k+1})  - f^* \rt]  \ ,
	\end{eqnarray*}
	where the first inequality is from the monotonity of the iterations, the second inequality is from the contraction condition \eqref{contraction-condition}, and the third  inequality is from assumption~\eqref{assumption-a}.
	
	\smallskip
	
	\noindent
	{\bf (Case $ii$)} Here we consider the case when 
	\begin{eqnarray}\label{ass-case2}
		F(w^{k+1}) - f^* > (1-\tau) \lt[ 	\bar f(u^k,w^{k+1})  - f^* \rt]	
	\end{eqnarray}
	which is equivalent to
	\begin{equation}\label{assumption-b}
		\bar f(u^k,w^{k+1}) - F(w^{k+1}) < \tau \lt[ 	\bar f(u^k,w^{k+1})  - f^* \rt] \ .
	\end{equation}
	Define $M_k :=  \bar f(u^k,w^{k+1})  - f^* $. Then, it holds that
	\begin{eqnarray}
		\tau M_k &\mathop{>}\limits^{(i)}&   \bar f(u^k,w^{k+1}) - F(w^{k+1})  \nonumber\\
		&\mathop{\ge}\limits^{(ii)}&
		\bar f(u^{k+1},w^{k+1}) - F(w^{k+1})  \nonumber\\
		&\mathop{\ge}\limits^{(iii)}&
		\lt\la \na_u \bar f(u^*_{w^{k+1}},w^{k+1}) , u^{k+1} - u^*_{w^{k+1}}\rt\ra \nonumber\\
		&&
		+(1/(2L)) \lt\|  \na_w \bar f(u^{k+1},w^{k+1}) - \na_w    \bar f(u^*_{w^{k+1}},w^{k+1})   \rt\|^2_2 \nonumber\\
		&\mathop{\ge}\limits^{(iv)}&
		({1}/{(2L)}) \lt\|  \na_w \bar f(u^{k+1},w^{k+1}) - \na_w    \bar f(u^*_{w^{k+1}},w^{k+1})   \rt\|^2_2 \nonumber\\
		&=&
		({1}/{(2L)}) \lt\|  \na_w \bar f(u^{k+1},w^{k+1}) - \na F(w^{k+1})   \rt\|^2_2 \ , \nonumber
	\end{eqnarray}
	where inequality $(i)$ is from \eqref{assumption-b}, $(ii)$ is from monotonicity of the iterations, $(iii)$~is from Lemma \ref{lemma:smooth-inequality}, and $(iv)$~is because of the optimality of $u^*_{w^{k+1}}$. The final equality is from Lemma \ref{lemma:diff-under-inf} (1). As a result, we have 
	$$
	\|  \na_w \bar f(u^{k+1},w^{k+1}) - \na F(w^{k+1})  \|_2 \le 
	\sqrt{2L\tau M_k} \ ,
	$$
	and hence
	\begin{eqnarray}
		\|\na_w \bar f(u^{k+1},w^{k+1})\|_2
		&\ge&
		\|\na F(w^{k+1})\|_2 -  \|\na F(w^{k+1})-\na_w \bar f(u^{k+1},w^{k+1}) \|_2 \nonumber\\
		&\ge &
		\|\na F(w^{k+1})\|_2   -	\sqrt{2L\tau M_k} \nonumber\\
		& \mathop{\ge}\limits^{(i)} &
		\sqrt{2\mu (F(w^{k+1}) - f^*)} -	\sqrt{2L\tau M_k}   \nonumber\\
		&\mathop{\ge}\limits^{(ii)}&
		\sqrt{2\mu (1-\tau ) M_k} - 	\sqrt{2L\tau M_k} \nonumber\\
		&=&
		\lt(   \sqrt{2\mu(1-\tau)} - \sqrt{2L\tau}  \rt) \sqrt{ M_k} \ , \nonumber
	\end{eqnarray}
	where $(i)$ uses Lemma~\ref{lemma:diff-under-inf} (3), and $(ii)$ is from inequality~\eqref{ass-case2}. 
	As $ \tau = {\mu\eta}/{4} $, we have 
	$$
	\sqrt{2\mu(1-\tau)} - \sqrt{2L\tau} \ge
	\sqrt{2\mu (3/4)} - \sqrt{2L\eta ({\mu}/{4})} \ge
	\lt( \sqrt{ 3/2} -  \sqrt{ 1/2}\rt) \sqrt{\mu} > 
	\sqrt{\mu}/2\ ,
	$$
	and thus 
	\begin{eqnarray}\label{ineq8}
		\|\na_w \bar f(u^{k+1},w^{k+1})\|_2 \ge (1/2) \sqrt{\mu(\bar f(u^k,w^{k+1}) - f^* )}  \ .
	\end{eqnarray}

	On the other hand, 
	by the gradient step in the unbounded subspace, we have
	$$
	\bar f(u^{k+1},w^{k+2}) \le
	\bar f(u^{k+1},w^{k+1}) - ({\eta}/{2}) \|\na_w \bar f(u^{k+1},w^{k+1})\|_2^2 \ .
	$$
	Therefore,
	$$
	\bar f(u^{k+1},w^{k+1}) - \bar f(u^{k+1},w^{k+2})
	\ge
	\frac{\eta}{2} \|\na_w \bar f(u^{k+1},w^{k+1})\|_2^2
	\ge
	\frac{\eta\mu}{8} (\bar f(u^k,w^{k+1}) - f^* ) \ ,
	$$
	where the last inequality is from \eqref{ineq8}.
	{\color{black}
		Noticing the algorithm is descent, it holds
		\begin{eqnarray}
			\bar f(u^{k},w^{k+1}) - \bar f(u^{k+1},w^{k+2}) &\ge& \bar f(u^{k+1},w^{k+1}) - \bar f(u^{k+1},w^{k+2}) \nonumber\\
			&\ge&
			({\mu \eta}/8) (\bar f(u^k,w^{k+1}) - f^* ). \nonumber
		\end{eqnarray}
		The above leads to
		$$
		\bar f(u^{k+1},w^{k+2}) - f^* \le
		(1-\frac{\mu\eta}{8}) (\bar f(u^k,w^{k+1}) - f^* ) \le 	\lt(   1- \rho \frac{\mu \eta}{4}   \rt) (\bar f(u^k,w^{k+1}) - f^* ) \ ,
		$$
		where the last inequality is because $\rho = \min\{1/2, \ga\} \le 1/2$. 
	}
\end{proof}

Now we are ready to prove Theorem \ref{CGG-linear-conv} by combining Propositions \ref{FW-contraction} and~\ref{prop:combining}:

\medskip

\noindent
\textbf{Proof of Theorem \ref{CGG-linear-conv}:} In view of Proposition \ref{FW-contraction}, among the first $k$ iterations, at least $k/2$ of them satisfy condition \eqref{contraction-condition}.
Proposition~\ref{prop:combining} states that when~\eqref{contraction-condition} holds we also have~\eqref{contraction-result}, hence \eqref{contraction-result} holds for at least $k/2$ of the first $k$ iterations.
Meanwhile, notice that uAFW is a descent algorithm, therefore we have
$$
\bar f(u^k,w^{k+1}) - f^* \le \lt(1- (\rho{\mu \eta}/{4}) \rt)^{k/2} (\bar f(u^0,w^1) - f^*) \ .
$$
We complete the proof of Theorem \ref{CGG-linear-conv} by substituting $\rho=\min \{1/2, \sigma/\bar C_{f,x^0}^S\}$ in Proposition \ref{FW-contraction}. \qed

\section{Applications}\label{section: applications}

We discuss how to efficiently apply uFW (Algorithm \ref{CG-G method}) and uAFW (Algorithm \ref{CG-G with away steps}) to solve the two motivating applications mentioned in Section~\ref{sec:intro}.
Section~\ref{subsection: A Class of Constrained Supervised Learning Problem} discusses other applications of our framework.

\subsection{The $\ell_1$ trend filtering problem (of order $r$)}\label{trend-filtering-subsection}

The $\ell_1$ trend filtering problem~\eqref{intro:generalized-lasso} of order $r$ is an instance of the following optimization problem:
\begin{equation}\label{trend-filtering}
	\min_{x\in \R^n} ~ f(x) ~~~\stt ~~~ \ \|D_n^{(r)} x\|_1 \le \delta \ ,
\end{equation}
where $f$ is a smooth convex function; $D^{(r)}_n$ denotes the $r$-order discrete derivative operator, 
that is, $D^{(1)}_n$ is a matrix in $\R^{(n-1)\times n}$ with $(i,i)$-entry being $1$ and $(i, i+1)$-entry being $-1$ (for $i\in [n-1]$) and all other entries being $0$. For $r\ge 2$, it is defined recursively by
$
D^{(r+1)}_n = D^{(1)}_{n-r} \cdot D^{(r)}_n$.

Popular instances of $\ell_1$ trend filtering consider small values of $r$. $\ell_1$ trend filtering with $r=1$ is also called the fused lasso~\cite{fused-lasso}, where the constraint $\|D^{(1)}_n x\|_1 \le \delta$ tends to induce piecewise constant solutions in $x$. $\ell_1$ trend filtering with $r=2$ and $r=3$ favor piecewise linear solutions and piecewise quadratic solutions, respectively.

Trend filtering problems have been extensively studied in the literature. There have been many approaches for solving it, for example, interior point method \cite{Boyd-trend-filter}, ADMM type methods \cite{ramdas2016fast}, and pathwise algorithms \cite{gen-lasso}. There is also an efficient dynamic programming based method for the case when $r=1$~\cite{johnson2013dynamic}. Most of the above approaches work only for a specific form of the objective function $f$. In the following, we show how to utilize uFW and uAFW to solve this problem with a general convex smooth objective $f$.  

\smallskip

\noindent \textbf{Writing \eqref{trend-filtering} in the form of \eqref{problem1}}:
We can write~\eqref{trend-filtering} in the form of~\eqref{problem1} with
$$
T = \ker(D^{(r)}_n), \quad S = 
\Big\{  x\in \R^n ~\big|~ \|D^{(r)}_n x\|_1 \le \delta , \ x \in \big(\ker(D^{(r)}_n) \big)^\perp \Big\}\ .
$$

\subsubsection{Computing projections onto $T$ and $T^\perp$}
The following lemma gives
a characterization of $T = \ker(D^{(r)}_n)$.
\begin{lemma}\label{natrual-basis}
	Let $U \in \R^{n\times n}$ be an upper triangular matrix with all its upper triangular entries being $1$. Then 
	$$T=\ker(D^{(r)}_n) = \spn \{\bsone_n, U \bsone_n, U^2 \bsone_n, \dots, U^{r-1}\bsone_n\} \ .$$
\end{lemma}
\begin{proof}
	It is straightforward to verify the following identity (e.g., by induction)
	\begin{equation}\label{HL}
		D_{n}^{(i)} U^{i} = [I_{n-i},\boldsymbol{0}_{(n-i) \times i}], \ \ \ \forall 1\le i \le r.
	\end{equation}
	Hence for all $1\le i \le r-1$, it holds $D_n^{(r)} U^i \bsone_n   = 0$, which implies 
	$$ \spn \{\bsone_n, U \bsone_n, U^2 \bsone_n, \dots, U^{r-1}\bsone_n\} \subseteq  \ker(D^{(r)}_n) \ .$$
	On the other hand, assume that $\sum_{i=0}^{r-1} \al_i U^i \bsone_n=0$ for some $\al_0, \dots, \al_{r-1}\in\R$. Multiplying $D_n^{(r-1)}$ in both sides and in view of $D_n^{(i+1)} U^i \bsone_n = 0$, we have $\al_{r-1} = 0$. Similarly, we can prove $\al_{r-2} = \cdots = \al_{1} = 0$; and therefore $\{\bsone_n, U \bsone_n, U^2 \bsone_n, \dots, U^{r-1}\bsone_n \}$ are linearly independent. Note also that $\dim(\ker(D^{(r)}_n)) = r$, thereby completing the proof of the lemma.
\end{proof}

Let  $A:= [\bsone_n, U \bsone_n, \cdots, U^{r-1} \bsone_n] \in \R^{n\times r}$. To compute the projection operator $\cP_T$, we perform a QR decomposition of $A=QR$ where $Q$ is a matrix with orthogonal columns and the same size of $A$, and $R$ is an upper triangular square matrix. Note that the columns of $Q$ form an orthogonal basis of subspace $T$, and the projections onto $T$ and $T^{\perp}$ are given by 
$\cP_{T}x = QQ^\T x$ and $\cP_{T}^{\perp} x = x - \cP_{T}x.$

\subsubsection{Solving the linear programming subproblem}

The linear oracle in uFW (Algorithm \ref{CG-G method}) 
has the following form
\begin{equation}\label{sub1}
	\min_{x}~ \langle c,x \rangle ~~~ \stt ~~ \|D^{(r)}_n x\|_1 \le \delta , \ x \in (\text{ker}(D^{(r)}_n) )^{\perp} \ ,
\end{equation}
where $c=\nabla f(y^k)$ in the $k$-th iteration. 
As we discuss below, \eqref{sub1} admits a closed-form solution. To this end, 
note that~\eqref{sub1} can be written as
\begin{equation}\label{sub2}
	\min_{x} ~\langle c,x \rangle ~~~ \stt ~~ \|D^{(r)}_nx\|_1 \le \delta , \ Q^\T x = 0 \ ,
\end{equation}
where $Q\in \R^{n\times r}$ is computed by the QR decomposition as stated above.

We apply a variable transformation via $x = U^r y$. 
Now define $\bar c := (U^r)^\T c = [\bar c_1; \bar c_2]
$, where 
$\bar c_1 \in \R^{n-r}$ is the first $n-r$ components of $\bar c$ and $\bar c_2 \in \R^r$ is the last $r$ components of $\bar c$. 
Similarly, define $y = [z; w ]
$ where $z  \in \R^{n-r}$ is the first $n-r$ components of $y$ and $w \in \R^r$ is the last $r$ components of $y$. We let $B_1\in \R^{r \times (n-r)}$ and $B_2 \in \R^{r\times r}$ be such that $Q^\T U^r = [B_1,B_2]$. Then problem \eqref{sub2} can be converted to
\begin{equation*}
	\min_{z,w} \ \ \la \bar c_1 , z \ra + \la \bar c_2 , w \ra \quad
	\stt ~~ \|z\|_1 \le \delta , \ B_1 z + B_2 w = 0\ .
\end{equation*}
From the equality constraints, we have $w = -B_2^{-1} B_1 z$ 
and the problem reduces to
\begin{equation}\label{sub5}
	\min_{z} \ \la \bar c_1 - B_1^\T (B_2^{-1})^{\T} \bar c_2 , z \ra \quad \stt ~~
	\|z\|_1 \le \delta \ .
\end{equation}
Indeed, \eqref{sub5} has a closed-form solution.
Denote $\wtd c :=  \bar c_1 - B_1^\T (B_2^{-1})^{\T} \bar c_2$, and let $j=\argmax_i |\wtd c_i|$ be its largest component in absolute value, then the solution of \eqref{sub5} is 
$
z = -\text{sign}(\wtd c_j) \delta e_j
$, that is, if $\wtd c_j\ge 0 $, $z = -\delta e_j$; otherwise $z = \delta e_j$. 
Reversing the above process, we obtain the solution of problem \eqref{sub1}.

For uAFW, we need another linear oracle to compute  $v^k$ in step (2) of Algorithm \ref{CG-G with away steps}. This can be computed in a manner similar to that outlined above. Note that the vertex set $V(x^k)$ is a subset of vertices of $S$. 
As above, this problem can also be transformed to a similar form as \eqref{sub5} with a fewer number of variables.

The cost of computing the linear oracle above is $O(nr)$. In addition, there is a one-time cost of $O(nr^2)$ for computing the QR decomposition of $A$ and $B_1^\T (B_2^{-1})^{\T}$.

\subsection{Matrix Completion with Generalized Nuclear Norm Constraints}\label{subsection:generalized-nuclear-norm}
The nuclear norm is often used as a convex surrogate for the rank of a matrix and arises in the matrix completion problem \cite{candes2009exact,Mazumder2010spectral,Freund-inface}. 
Here we consider a family of matrix problems with generalized nuclear norm constraints, and the goal is to solve 
problem \eqref{intro:generalized-nuclear-norm}
where, $P\in \R^{k\times m}$, $Q\in \R^{n\times l}$ are given matrices
that encode prior information about the problem under consideration\footnote{We assume that there exists a finite optimal solution to \eqref{intro:generalized-nuclear-norm}.}.

In spite of its wide applicability~\cite{angst2011generalized,srebro2010collaborative,Fithian2018fexible}, solving~\eqref{intro:generalized-nuclear-norm} is computationally expensive  compared to the case when $P=I$ and $Q=I$, which is amenable to the usual Frank-Wolfe method~\cite{jaggi2013revisiting,Freund-inface}.
Below we show that our proposed framework suggests a useful approach for~\eqref{intro:generalized-nuclear-norm}---in particular, we need to compute the largest singular value/vector pair at every iteration, which can be performed efficiently via iterative methods having a low-computational cost per iteration~\cite{Golub1996matrix,Freund-inface}. 

\smallskip

\noindent \textbf{Writing \eqref{intro:generalized-nuclear-norm} in the form of \eqref{problem1}}:
Define the linear operator $\varphi_{P,Q} : X \mapsto PXQ$, then the feasible set of problem \eqref{intro:generalized-nuclear-norm} becomes $\mathcal{F} := \{X\in \R^{m\times n} ~|~ \|\varphi_{P,Q}(X)\|_* \le \delta\}$. Thus, we can convert \eqref{intro:generalized-nuclear-norm} to \eqref{problem1} with 
$$T= \ker(\varphi_{P,Q}) ~~ \text{and} ~~S=
\{X\in \R^{m\times n}~|~ X \in  \ker(\varphi_{P,Q})^{\perp}, \|PXQ\|_* \le \delta\}\ .$$

\subsubsection{Projections onto $T$ and $T^\perp$}

Recall that $\mc{P}_T$ and $\cP_{T}^\perp$ denote the projections onto spaces $T$ and $T^{\perp}$ respectively.
Note that the matrix representation of
$\varphi_{P,Q}$ is $Q^\T \otimes P$ (Kronecker product). 
Notice
$$T^\perp = \ker(Q^\T \otimes P)^\perp = \ran (Q\otimes P^\T)\ ,$$
thus the matrix representation of $\mathcal{P}_{T^\perp}$ can be written as
$$
(Q \otimes P^\T) (Q \otimes P^\T)^+ = 
(Q \otimes P^\T) (Q^+ \otimes P^{+\T}) =
(QQ^+) \otimes (P^\T P^{+\T}) = 
(QQ^+) \otimes (P^+P) \ ,
$$
where $P^{+} $ and $Q^+$ denote the Moore-Penrose inverses of $P$ and $Q$ (respectively). 
Therefore, we have $\mathcal{P}_{T}^\perp(X) = P^+P X QQ^+$ and $\mathcal{P}_{T}(X) =  X - P^+P X QQ^+ $.

\subsubsection{Solving the linear subproblem}

We show that the linear subproblem in the uFW framework for problem \eqref{intro:generalized-nuclear-norm} can be solved via computing the leading singular vector of a matrix.  Note that in step (2) of the uFW algorithm, we need to solve a problem of the following form
\begin{equation}\label{nu-sub}
	\min_{X} ~~\la C,X \ra \quad \stt~~   X \in  T^{\perp},\  \|PXQ\|_* \le \delta \ ,
\end{equation}
where $C$ is a given matrix in $\R^{m\times n}$.
Lemma~\ref{nu-subproblem-opt} shows how to solve~\eqref{nu-sub}. 

\smallskip

\begin{lemma}\label{nu-subproblem-opt}
	Let $\bar C := (P^+)^\T C (Q^+)^\T$. 
	Suppose the SVD of $\bar C$ is $\bar C = U D V^\T$, where $ U\in \R^{m\times m}$ and $V \in \R^{n\times n}$ are orthogonal matrices, 
	and $D$ is in $\R^{m\times n}$ with $D_{ii} = \sigma_i$ for $i=1,...,r$ ($r$ is the rank of $\bar C$) and all other entries being $0$, satisfying $\sigma_1 \ge \cdots \ge \sigma_r >0$.
	Let $u_1$ be the first column of $U$, $v_1$ be the first column of $V$. Then, the rank-one matrix $X^* := - \delta P^+ u_1 v_1^\T Q^+$ is an optimal solution to problem \eqref{nu-sub}.
\end{lemma}

\begin{proof}
	First, we show that $X^*$ satisfies the constraints. Note that
	$$
	T^\perp = \ker(Q^\T \otimes P )^\perp = \ran[(Q^\T \otimes P)^+] = \ran [(Q^+)^\T \otimes P^+] \ ,
	$$
	which means $Y \in T^\perp$ if and only if $\exists Z\in \R^{m\times n}$ such that $Y = P^+ Z Q^+$. Therefore, we have $X^* \in T^\perp$. In addition, we have
	\begin{eqnarray}
		\|PX^* Q\|_* = \dt \| PP^+ u_1 v_1^\T Q^+Q \|_* =\dt
		| v_1^\T Q^+Q PP^+ u_1 |\ , \nonumber
	\end{eqnarray}
	and hence
	$$
	\|PX^* Q\|_* \le \dt
	\|Q^+ Q v_1\|_2 \|PP^+ u_1\|_2 \le \dt
	\|v_1\|_2 \|u_1\|_2 = \dt \ ,
	$$
	where the second inequality is because $PP^+$ and $Q^+Q$ are projection matrices.
	
	Now we verify that $X^*$ is an optimal solution. First, note that 
	$$
	\la C , X^* \ra = -\delta \la C, P^+ u_1 v_1^\T Q^+ \ra =
	-\dt u_1^\T (P^+)^\T C (Q^+)^\T v_1 
	= -\dt u_1^\T \bar C v_1 =
	-\dt \sigma_1 \ .
	$$
	Secondly, for any $X$ satisfying $X\in T^\perp$ and $\|PXQ\|_* \le \dt$, we have $X = \cP_{T}^\perp X = P^+PXQQ^+$, and 
	$$
	\la X,C \ra =
	\la P^+PXQQ^+, C \ra =
	\la PXQ, (P^+)^\T C (Q^+)^\T \ra =
	\la PXQ ,\bar C \ra \ . 
	$$
	Therefore, 
	$$
	\la X,C \ra  \ge
	-\sigma_1 \|PXQ\|_* \ge
	-\sigma_1 \dt \ ,
	$$
	where the first inequality is because the spectral norm and the nuclear norm are dual of each other. This finishes the proof.
\end{proof}

\begin{remark}
	In several applications~\cite{Fithian2018fexible}, the matrices $P$ and $Q$ are projection 
	matrices (i.e., $P^+ = P$ and $Q = Q^+$).
	In such cases~\eqref{nu-sub} can be computed more efficiently.
	Using the same notations as the statement of Lemma \ref{nu-subproblem-opt},
	let $\bar C = P CQ$ and $u_1$ and $v_1$ be the leading (left and right) singular vectors of $\bar C$. Then an optimal solution of \eqref{nu-sub} is given by $X^* = -\dt u_1 v_1^\T$. 
\end{remark}

Note that $u_1$ and $v_1$ can be computed by iterative procedures \cite{Golub1996matrix} for large problems. When $P$, $Q$ and $C$ are sparse or low-rank, we can further explore these structures via the matrix-vector multiplication steps within the iterative procedure. This makes our algorithm computationally efficient.

\subsection{Other Applications}\label{subsection: A Class of Constrained Supervised Learning Problem}

Our presented algorithms (uFW and uAFW) can be generalized to solve the following family of learning problems:
\begin{equation}\label{gen-learn-model-cons}
	\min_{x} ~~~ \sum_{i=1}^N h \lt(b^i,  a^i, x \rt) ~~\stt ~~ \varphi(H x) \le \delta \ ,
\end{equation}
where $a^i\in\R^n$ is the feature vector and $b^i\in \R$ is the corresponding response of the $i$-th sample, $x\in \R^n$ is the model parameter to be optimized, $h$ is a given smooth convex loss function, $H$ is a given matrix 
and $\varphi$ is a nonnegative function. When $\varphi$ is a norm and $H$ does not have full row rank, the constraint set in \eqref{gen-learn-model-cons} is unbounded. The constraint set can be written in the form $T\oplus S$ with $T= \ker(H)$ and 
$S= \{ x\in \R^n ~|~ \varphi(H x) \le \delta, \ x \in (\ker(H))^\perp \}$, hence our algorithms uFW/uAFW can be applied. Concrete examples of this framework include total variation denoising~\cite{rudin1992nonlinear}, group fused lasso~\cite{alaiz2013group} and convex clustering~\cite{pelckmans2005convex}.

\section{Numerical Experiments}\label{section: experiments}
We present numerical experiments with uFW and uAFW for Problems~\eqref{intro:generalized-lasso} and~\eqref{intro:generalized-nuclear-norm}. Our code is available at

\href{https://github.com/wanghaoyue123/Frank-Wolfe-with-unbounded-constraints}{https://github.com/wanghaoyue123/Frank-Wolfe-with-unbounded-constraints}.

\subsection{$\ell_1$ trend filtering}

We first consider the $\ell_1$ trend filtering problem. 

\smallskip

\noindent
{\bf Data generation.} We generate a Gaussian ensemble $A= [a_1, a_2 , ..., a_n]^\T \in \R^{N\times n}$ with 
iid $N(0,1)$ entries, and
the noise is $\ep_i\sim N(0,\sigma^2)$ with variance $\sigma^2 >0$ for $i\in [N]$. 
The underlying model coefficient $x^*\in\R^n$ is generated as a piecewise constant vector when $r=1$ and a piecewise linear vector when $r=2$. 
More precisely, for $r = 1$, $x^*$ is piecewise constant with $5$ pieces, each of equal length.
The value of each piece is generated from the uniform distribution on $[-1/2,1/2]$ and then normalized such that $\| D^{(1)}_n x^*\|_1 = 1$. 
For $r=2$, $x^*$ is piecewise linear with $5$ pieces. These $5$ pieces have equal lengths, and the slope of each piece is generated from the uniform distribution on $[-1/2,1/2]$ and then normalized such that $\| D^{(2)}_n x^*\|_1 = 1$.
The response is then obtained from the following linear model: $b_i =  a^T_i x^* + \ep_i,~i\in [N]$.
To estimate $x^*$, we consider the following problem:
\begin{eqnarray}\label{eq:numerical-trend}
	\min_{x} ~\| b- A x \|_2^2  ~~~~ {\rm s.t.}  ~~ \| D^{(r)}_n x \|_1 \le \dt \ ,
\end{eqnarray}
where the bound $\dt$ is taken as $\|D^{(r)}_n x^*\|_1$.
	We define the \textit{Signal-noise-ratio (SNR)} of the problem as $ SNR := \| Ax^*\|^2/(n\sigma^2)$.

\smallskip

\noindent {\bf{Computational environment}.}	Our algorithms are written in Python 3.7.4, and we compare with MOSEK~\cite{andersen2000mosek} and ADMM-based solver SCS~\cite{scs,o2016conic} (SCS for short) by calling the solvers through CVXPY~\cite{diamond2016cvxpy}.
Computations were performed on MIT Sloan's engaging cluster with 4 CPUs and 8GB RAM per CPU.
The reported results are averaged over three independent trials.

\smallskip

\noindent 
{\bf{Performance of uFW and uAFW}.} 
For the sequence $\{x^k\}$ produced by our algorithms, let $f_k$ be the lowest objective value computed in the first $k$ iterations, and $G_k, H_k$ be the values defined in \eqref{def: Gk and Hk}. 
	We let the (relative) \textit{$G_k$-gap} be $G_k/{\max\{1,|f_k|\}}$ and (relative) \textit{$H^2_k$-gap} be $H_k^2/{\max\{1,|f_k|\}}$---they measure algorithm progress and can be used to design termination criteria for our algorithms.
	Additionally, we measure (relative) \textit{optimality gap} defined as ${(f(x^k) - f^*)}/{\max\{1,|f^*|\}}$, where $f^*$ is the optimal objective value computed by MOSEK. The relative measurement adjusts for the scale of the problem, and is commonly used in first-order solvers, such as SCS.

	Figure \ref{figure: comparison of different methods} presents the optimality gap, $G_k$-gap and $H_k^2$-gap 
	of uFW and uAFW versus the number of iterations. We consider~\eqref{eq:numerical-trend} for $r\in \{1, 2\}$. The data is generated as above with $N = 1000$, $n = 500$ and $SNR = 1$.
We compare the following methods:

\smallskip

$\bullet$ \textit{uFW (simple)}:  Algorithm \ref{CG-G method} with simple step-size rule \eqref{simple step-size rule}.

$\bullet$ \textit{uFW (linesearch)}: Algorithm \ref{CG-G method} with line-search rule \eqref{line-search step-size rule}.

$\bullet$ \textit{uAFW (linesearch)}: Algorithm \ref{CG-G with away steps} with line-search rule \eqref{line-search step-size rule}.

\smallskip

The corresponding gradient descent step-size $\eta$ is chosen as $1/\|A\|^2$ where $\|A\|$ is the operator norm of the data matrix $A$.

The left panel ($r=1$) in Figure \ref{figure: comparison of different methods} suggests that the away step variant: uAFW converges linearly to an optimal solution for the $\ell_1$ trend filtering problem with $r=1$, and it can reach very high accuracy. In contrast, uFW algorithms appear to converge at a sub-linear rate. Both these observations are consistent with our theoretical framework.  
The advantage of uAFW over uFW is less significant
in the right panel ($r=2$) in Figure \ref{figure: comparison of different methods}. We believe that this is due to the problem being very ill-conditioned for the case $r=2$---uAFW converges linearly, but with an unfavorable parameter for linear convergence that makes it progress slowly.   
Figure \ref{figure: comparison of different methods} also shows that uFW with line search step-size has a larger optimality gap compared to uFW with simple step-size. This is perhaps because the problem is ill-conditioned, and exact line-search may lead to conservative step-size. 
	However, the $G_k$-gap and $H_k^2$-gap of these methods are comparable. 
	Finally, note that in this example, $H_k^2$-gap is much smaller than $G_k$-gap and optimality gap. This is perhaps because the dimension of $T$ is low and the corresponding gradient steps in $T$ converge faster.

\smallskip

\noindent {\bf{Comparison with Benchmarks.}}
Table \ref{table: large instance r1}
compares uFW with two state-of-the-art conic optimization solvers:  MOSEK and SCS, for solving the $\ell_1$ trend filtering problem \eqref{eq:numerical-trend} with $r=1$ and $r=2$.
	We also compare with the accelerated projection free method for general unbounded constraint set~\cite{gonccalves2020projection} (denoted by APFA).

The data is generated by the procedure discussed above with $SNR=1$. 
	We use uFW with the simple step-size rule \eqref{simple step-size rule}, 
	and
	terminate uFW at iteration $k$ if $G_k/\max\{|f_k|,1\} <  10^{-4}$ and $H_k^2 /\max\{|f_k|,1\} < 10^{-4}$. 
	We run MOSEK and SCS with their default settings.
	For reference, Table~\ref{table: large instance r1} also reports the (relative) \textit{optimality gap}, denoted as uFW gap. 
	We set the termination tolerance of SCS as $10^{-3}$ and allow a maximum runtime of 2 hours.
		For APFA, we use Mosek to solve the linear oracles and the parameters are set following Section 4 of \cite{gonccalves2020projection}. Let $\hat f$ be the objective value computed by uFW upon termination, and $\wtd f_k$ be the objective value computed in iteration $k$ of APFA. 
		We terminate APFA at iteration $k$ if $(\wtd f_k - \hat f)/\max\{1,|\hat f|\} < 10^{-2}$, with a runtime limit of 2 hours.

	For each value of $r \in \{1, 2\}$, Table~\ref{table: large instance r1} presents results for three groups of problem sizes when $N\ll n$, $N\approx n$ and $N\gg n$. As we can see, the running time of uFW is significantly better than MOSEK and SCS for all the examples (one or two orders of magnitude better in most of the examples). 
	As the problem sizes grow, the runtimes of uFW appear to be stable. 
	MOSEK and SCS on the other hand, may fail to produce a reasonable solution (within the allocated maximum time-limit of 2hrs and/or memory). 
	We observe that when $G_k/\max\{|f_k|,1\} <  10^{-4}$ and $H_k^2 /\max\{|f_k|,1\} < 10^{-4}$, 
	the optimality gap of uFW is typically around $10^{-5} \sim 10^{-6}$. 
	Note that uFW takes longer to solve the problem with $r=2$ than the case $r=1$. This is because the problem with $r=2$ has a larger condition number.
	 It can be seen that the runtime of APFA is much longer than the other algorithms even for finding a low-accuracy solution. This is probably due to expensive linear optimization oracles\footnote{We need to minimize a linear function over the intersection of $T\oplus S$ and a Euclidean ball---as we are not aware of a simple (closed-form) solution to this, we solve this with Mosek.}.

	Finally, we note that MOSEK, SCS and uFW/uAFW are three different types of algorithms, each using different termination criteria. MOSEK is an interior-point solver, which naturally produces a high accuracy solution. SCS (based on ADMM) is a primal-dual first-order method, which may produce infeasible solutions, thus potentially negative primal-dual gaps (See Section 3.5 of \cite{o2016conic} for a detailed discussion on the termination criteria of SCS). uFW and uAFW are primal algorithms with feasible iterates, and the relative $(H_{k},G_{k})$-gaps provide a natural measure of solution quality. The runtime improvements in Table \ref{table: large instance r1} appear to suggest the advantage of uFW over competing methods.

\begin{figure}[h]
	\centering
	\scalebox{0.96}{\begin{tabular}{cc}
			$r=1$ & $r=2$\\
			\includegraphics[width=0.48\textwidth,trim={0 1.1cm 0 1cm},clip]{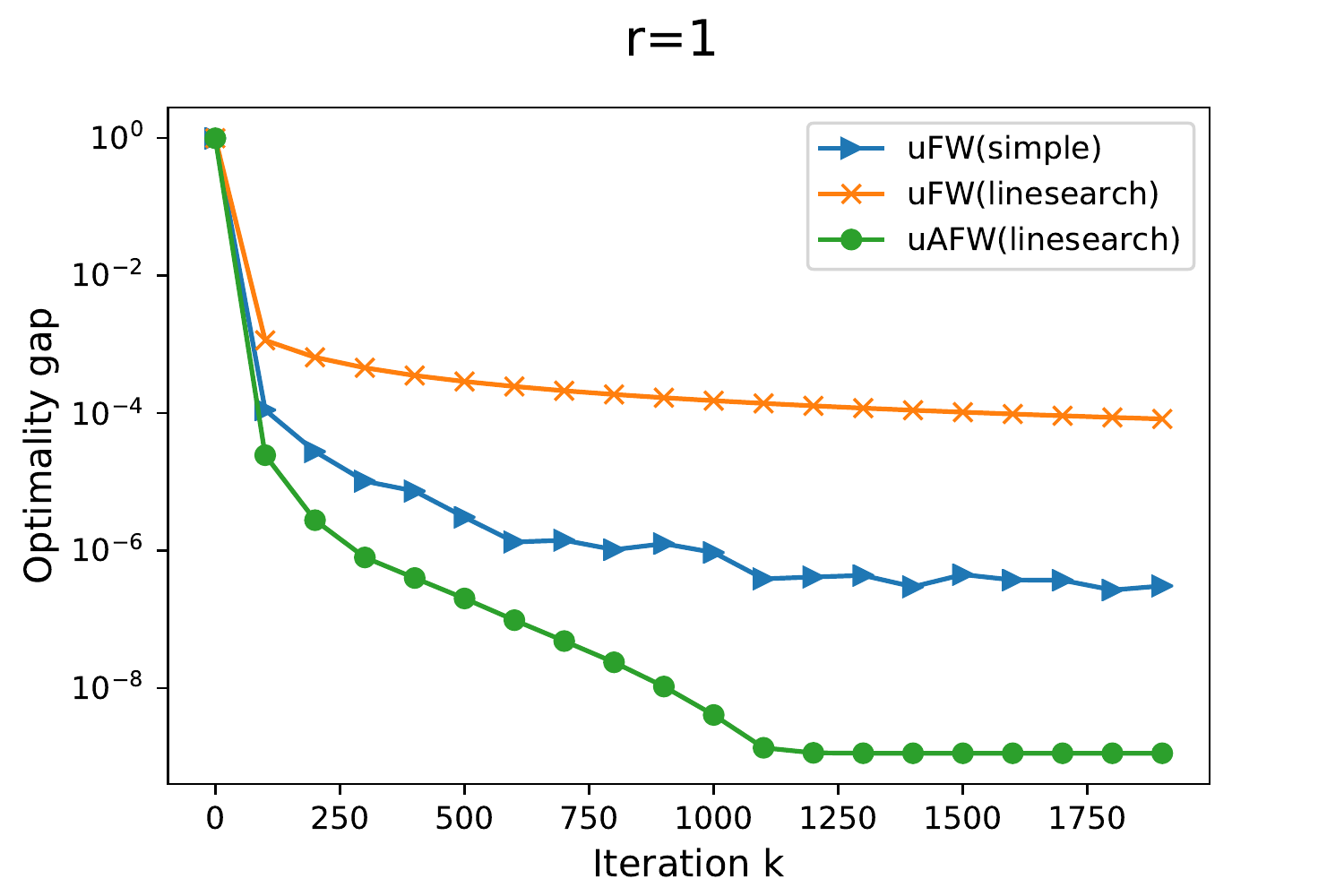} &
			\includegraphics[width=0.48\textwidth,trim={0 1.1cm 0 1cm},clip]{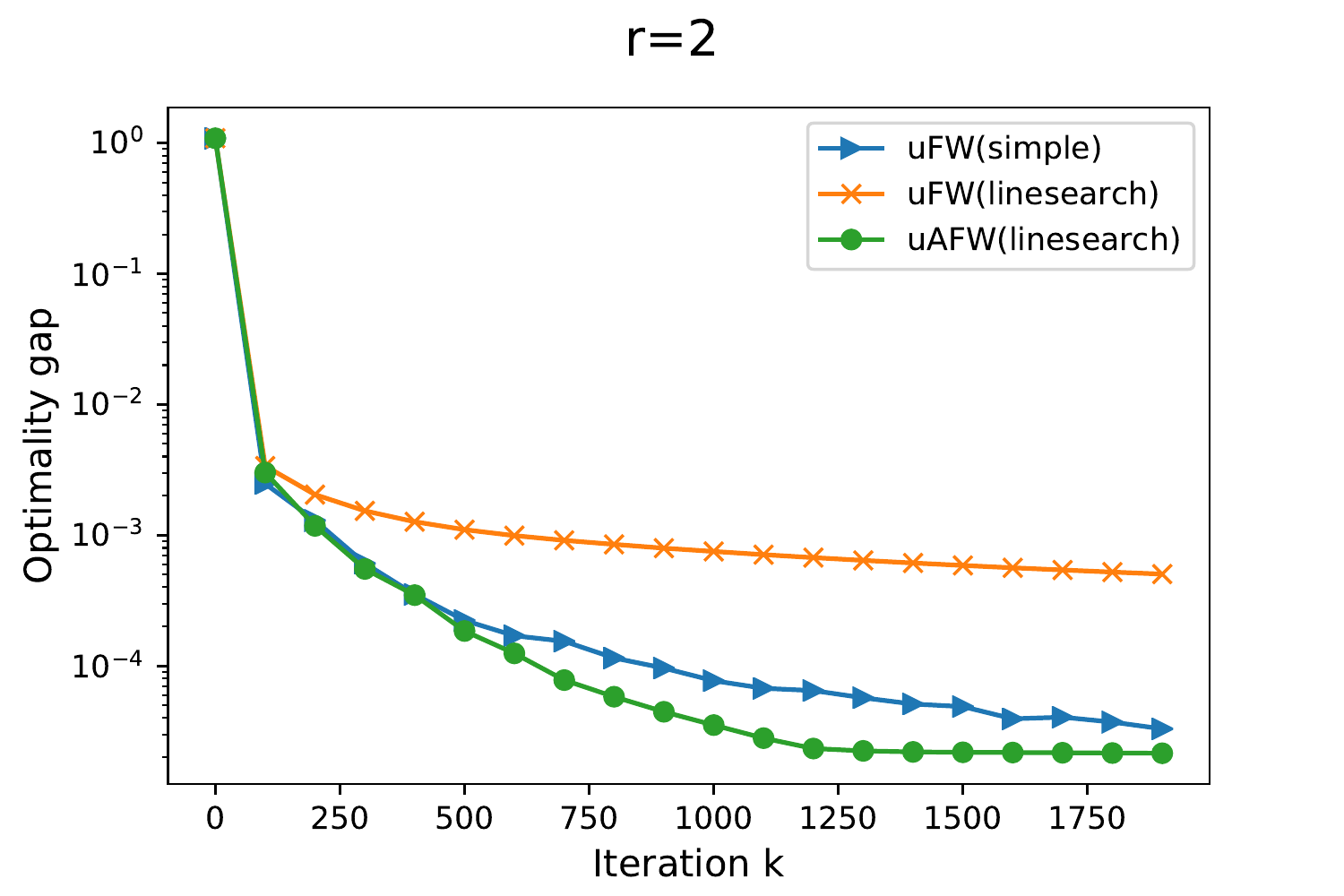} \\
			\includegraphics[width=0.48\textwidth,trim={0 1.1cm 0 1cm},clip]{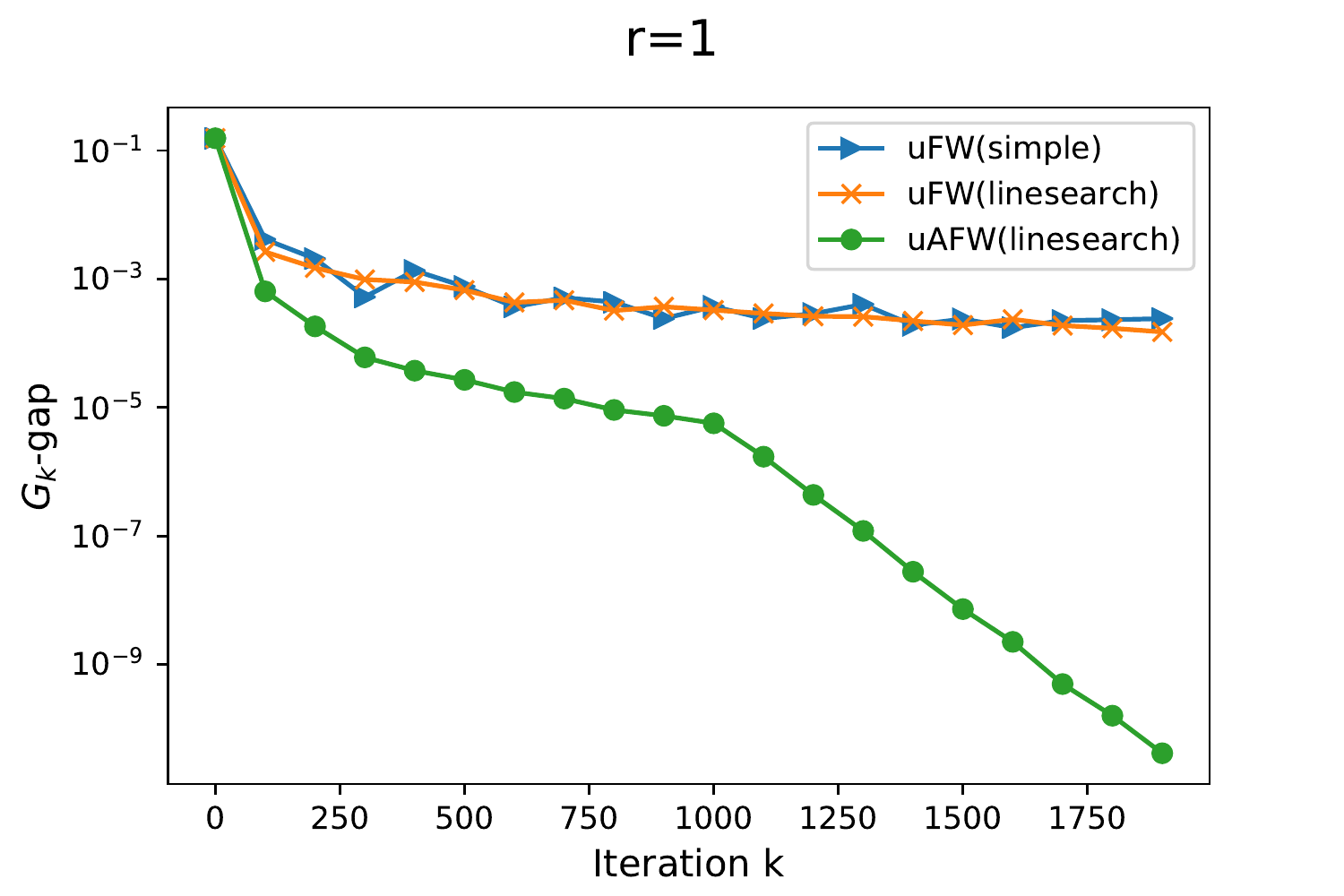} &
			\includegraphics[width=0.48\textwidth,trim={0 1.1cm 0 1cm},clip]{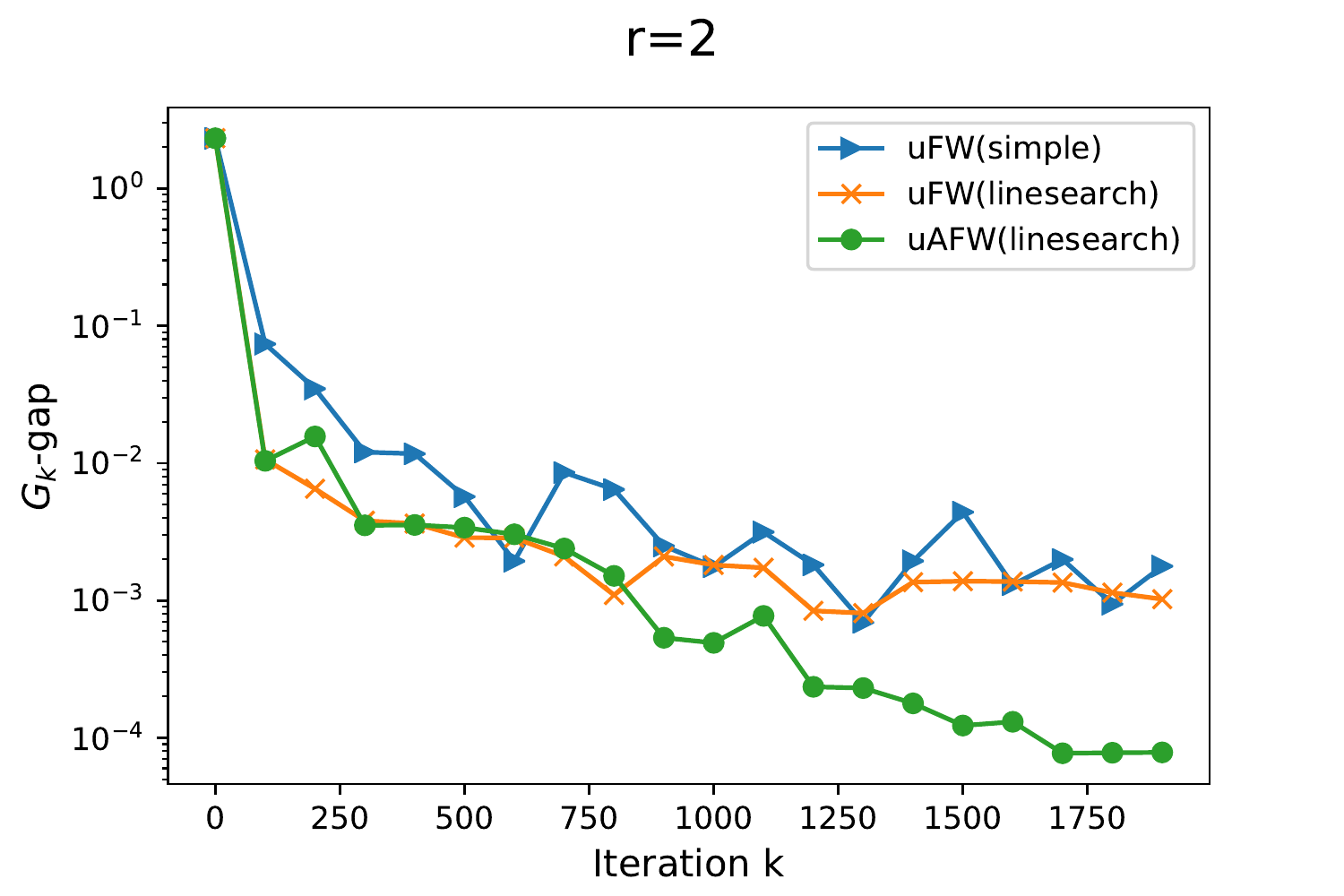} \\
			\includegraphics[width=0.48\textwidth,trim={0 0cm 0 1cm},clip]{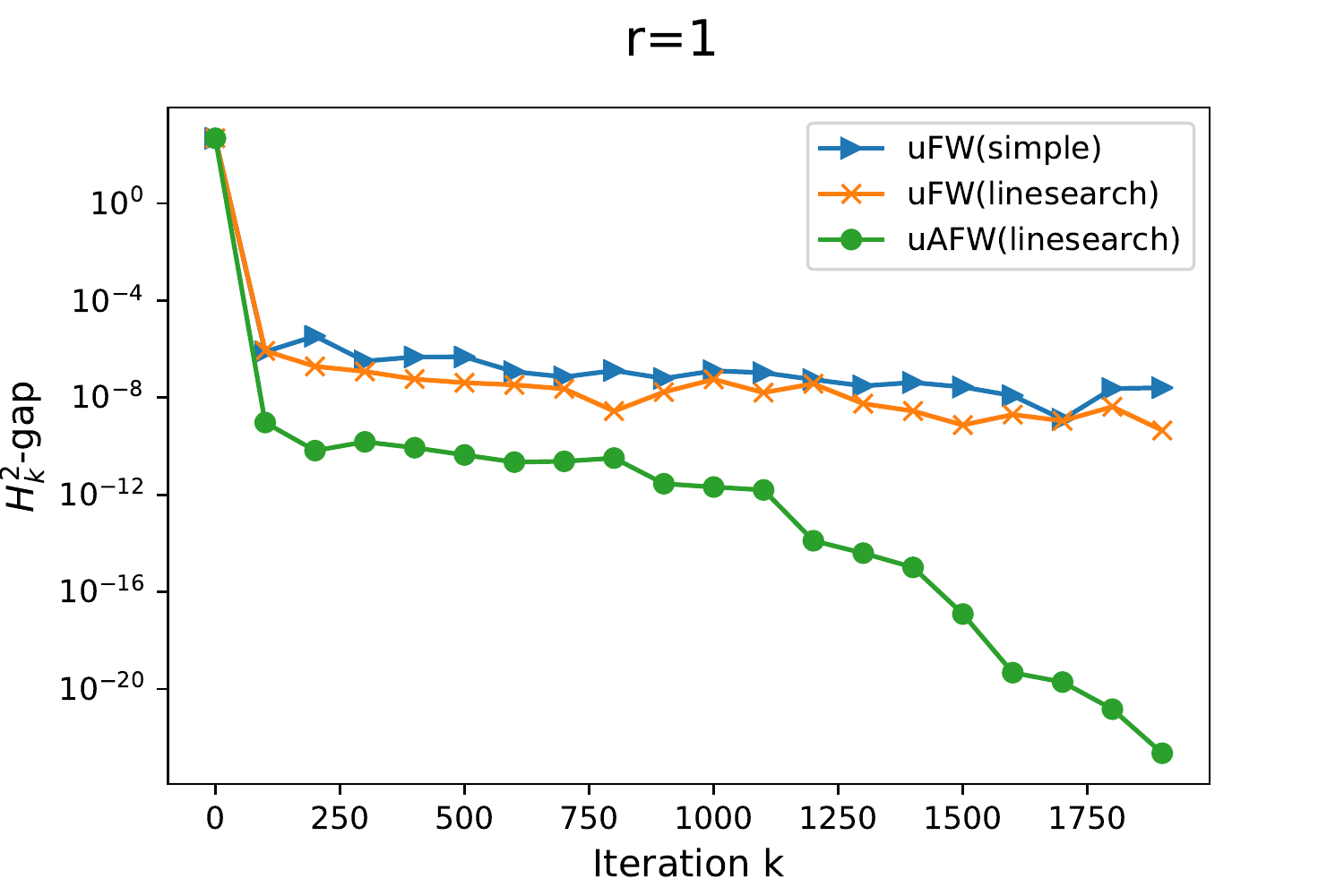} &
			\includegraphics[width=0.48\textwidth,trim={0 0cm 0 1cm},clip]{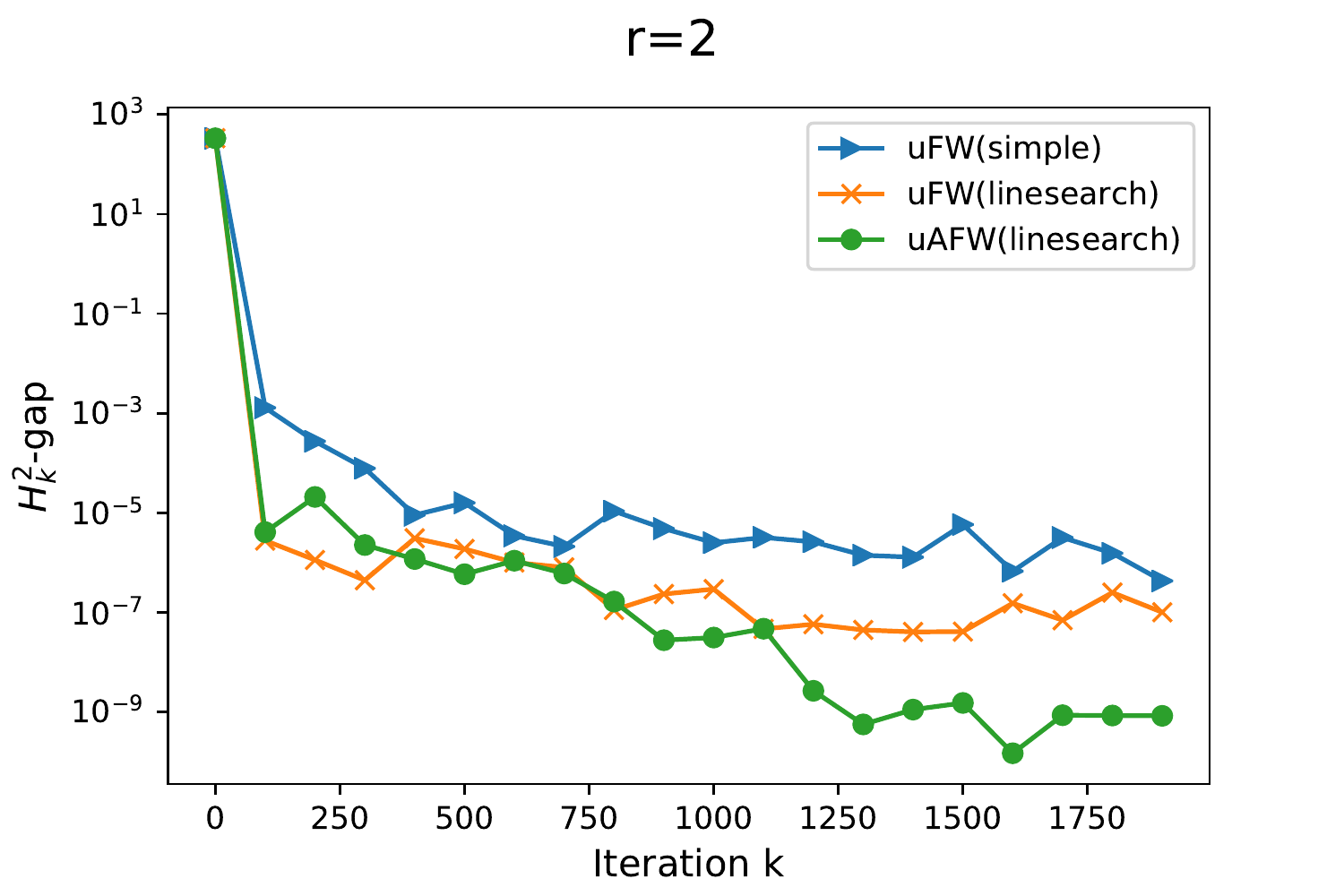} 
	\end{tabular}}
	\vspace{-1em}
	\caption{Plots showing the optimality gap, $G_k$-gap and $H^2_k$-gap obtained by uFW and uAFW versus the number of iterations (i.e., denoted by index $k$ in the algorithm descriptions) for solving the $\ell_1$ trend filtering problem \eqref{eq:numerical-trend} with $r=1$ and $r=2$, respectively.} 
	\label{figure: comparison of different methods}
\end{figure}

\begin{table}[h]
	\centering
	\caption{Table showing runtimes (in secs) for solving \eqref{eq:numerical-trend} by uFW, MOSEK and SCS with $r=1$ and $r=2$. For uFW we also display the optimality gap when $f^*$ is available. 
			The symbol ``-'' for MOSEK means the allocated memory was not sufficient. The symbol ``*'' means the method did not terminate in $2$ hours or the solver reported an error. 
			The symbol ``x'' means that MOSEK reported an error during execution.
	}
	\label{table: large instance r1}
	\scalebox{0.69}{
		\begin{tabular}{|c|c|ccccc|ccccc}
				\hline
				\multirow{3}{*}{N} & \multirow{3}{*}{n} & \multicolumn{5}{c|}{r=1}                                                 & \multicolumn{5}{c|}{r=2}                                                                      \\ \cline{3-12} 
				&                    & uFW     & uFW      & MOSEK   & SCS     & APFA & uFW     & uFW      & MOSEK   & SCS     & \multicolumn{1}{c|}{APFA} \\
				&                    & time(s) & gap      & time(s) & time(s) & time(s)                         & time(s) & gap      & time(s) & time(s) & \multicolumn{1}{c|}{time(s)}                         \\ \hline
				5K                 & 500                & 0.27    & 3.25e-07 & 3.42    & 10.70   & 226.17                          & 1.68    & 3.02e-06 & 2.39    & 53.19   & \multicolumn{1}{c|}{221.08}                          \\
				200K               & 1K                 & 7.62    & 7.04e-07 & 541.82  & 3879.35 & *                               & 28.82   & 2.98e-06 & 327.76  & 4114.53 & \multicolumn{1}{c|}{*}                               \\
				400K               & 1K                 & 17.48   & -        & -       & *       & *                               & 56.50   & -        & -       & *       & \multicolumn{1}{c|}{*}                               \\
				300K               & 2K                 & 21.32   & -        & -       & *       & *                               & 45.69   & -        & -       & *       & \multicolumn{1}{c|}{*}                               \\ \hline
				2K                 & 2K                 & 0.37    & 4.66e-07 & 7.17    & 43.14   & 939.62                          & 0.94    & 6.20e-06 & 7.74    & 97.12   & \multicolumn{1}{c|}{*}                               \\
				10K                & 10K                & 2.03    & 6.03e-07 & 302.14  & 2526.14 & *                               & 10.24   & 5.83e-07 & 448.85  & 3252.30 & \multicolumn{1}{c|}{*}                               \\
				15K                & 15K                & 4.36    & -        & -       & *       & *                               & 22.39   & -        & -       & *       & \multicolumn{1}{c|}{*}                               \\
				20K                & 20K                & 7.51    & -        & -       & *       & *                               & 48.14   & -        & -       & *       & \multicolumn{1}{c|}{*}                               \\ \hline
				500                & 5K                 & 0.23    & 1.14e-06 & 4.47    & 111.60  & *                               & 2.97    & 1.79e-06 & 4.72    & 109.12  & \multicolumn{1}{c|}{*}                               \\
				1K                 & 100K               & 6.29    & 2.45e-06 & 262.80  & 4763.89 & *                               & 48.29   & -        & x       & 4816.81 & \multicolumn{1}{c|}{*}                               \\
				1K                 & 200K               & 6.44    & -        & -       & *       & *                               & 101.54  & -        & -       & *       & \multicolumn{1}{c|}{*}                               \\
				2K                 & 300K               & 25.41   & -        & -       & *       & *                               & 173.91  & -        & -       & *       & \multicolumn{1}{c|}{*}                               \\ \hline
		\end{tabular}
	}
\end{table}

\subsection{Matrix Completion with Side Information}

Here we present computational results for the 
generalized matrix completion problem~\eqref{intro:generalized-nuclear-norm}. 

\smallskip

\noindent {\bf Data generation.}
The matrix of interest $B$ is generated from the model
\begin{equation}\label{MC-model}
	B = P_1 Z^\T + UV^{\T} + \cE
\end{equation} 
where $B\in \R^{m\times n}$, $P_1\in \R^{m\times r_1}$, $Z\in \R^{r_1 \times n}$, $U \in \R^{m\times r}$, $V\in \R^{n\times r}$ and $\cE \in \R^{m\times n}$. 
Here $P_1$ corresponds to the known side information about the column space of $B$, with coefficients $Z$; the low-rank component that is not included in this side information is $UV^\T$, and $\cE$ is the noise term. See \cite{Fithian2018fexible,eftekhari2018weighted} for more background on this statistical model.
The matrices $U, V, Z$ are generated independently with entries (iid) from $\cN(0,1)$. The matrix $P_1$ has unit-norm orthogonal columns, which is generated by taking the orthogonal basis for the columns of an $m\times r_1$ random matrix with iid standard Gaussian entries. The entries of $\cE$ are iid $\cN(0,\sigma^2)$.

We study the setting where only a subset of  the entries of $B$ are observed. 
Let $\Omega$ be the indices corresponding to the observed entries that are uniformly distributed across the matrix coordinates.
For a matrix $A\in \R^{m\times n}$, we let $\mathcal{P}_{\Omega} (A)$ be a matrix with entries in $\Omega$ being the same as $A$ and entries in $\Omega^c$ being zeros.
We estimate the signal ``$P_1 Z^{\T} + UV^{\T}$'' by solving the following problem (see \cite{Fithian2018fexible}):
\begin{equation}\label{MC-opt2}
	\min_{X}~ \|\mathcal{P}_{\Om} (X-B)\|_F^2~~~
	\stt~~~\|(I - P_1P_1^\T) X\|_* \le \delta \ ,
\end{equation}
which is a special case of~\eqref{intro:generalized-nuclear-norm} with $P = I_m - P_1 P_1^\T$, $Q= I_n$ and objective function $f(X) =\|\mathcal{P}_{\Om} (X-B)\|_F^2$ (here, $\| \cdot\|_F$ denotes the Frobenius norm).

\smallskip

\noindent
{\bf Computational environment.}
Our code is written in Matlab 2017, and we make use of Matlab built-in function \textit{svds} to compute the leading singular vector with a custom function for matrix-vector multiplication. We compare our proposal with SCS~\cite{o2016conic} (via CVX). 
Computations were performed on MIT Sloan's engaging cluster with 4 CPUs and 20GB RAM for each CPU. 
Our results are averaged over three independent experiments.

\smallskip

\noindent {\bf{Comparison with SCS.}} We compare the computation time of uFW and SCS on instances with $r=r_1=5$ and different values of $m$ and $n$, generated from the procedure stated above. 
We define the following three parameters for the experimental setting:

\noindent	$\bullet$ \textit{Signal-to-noise Ratio} (SNR): the ratio of signal variance vs noise variance, i.e., 
$$\text{SNR}:= \text{var}((P_1Z^\T + UV^\T)_{i,j}) / \text{var}(\cE_{i,j}) \ . $$ 

\noindent	$\bullet$
\textit{Non-zero ratio} (nnzr): the percentage of observed entries, i.e., $\text{nnzr} :=|\Omega | / (mn) $.

\noindent	$\bullet$
\textit{Relative diameter}~($\dt'$): the relative value of the parameter $\dt$ in \eqref{MC-opt2} with respect to the corresponding value of signal. That is, 
$$
\dt' := \dt / \| (I_m - P_1 P_1^\T ) UV^\T \|_* \ .
$$
We use uFW with the simple step size rule \eqref{simple step-size rule}. Let $f_k$ be the best objective value obtained in the first $k$ iterations of uFW. 
		We terminate 
	uFW when $G_k/\max\{|f_k|,1\} \le 3\times 10^{-3}$ and $H_k^2/\max\{|f_k|,1\} \le 3\times 10^{-3}$ and we set the SCS tolerance to be $3\times 10^{-3}$. For instances when SCS is able to output a solution, 
	we set $f^*$ as the objective value of the solution obtained by running SCS with a stricter tolerance $10^{-5}$, and define the (relative) \textit{optimality gap} of uFW as 
	$(f_k-f^*)/\max\{1,|f^*|\}$. 
	We also define the optimality gap of SCS as 
	$(\hat f-f^*)/\max\{1,|f^*|\}$, where $\hat f$ is the objective value for the SCS solution with tolerance $3\times 10^{-3}$. 
Note that SCS is an ADMM based solver, so the solution it obtains is not strictly feasible. To this end, we define the value SCS feasibility (\textit{SCS feas.}, in short) to be
$$
\text{SCS feas.}=({ \|  (I_m - P_1 P_1^\T )  \tilde{X} \|_* - \dt })/{\dt},
$$ 
where $\tilde{X}$ is the the solution obtained upon its termination. 

 Table \ref{table: MC} presents the comparison of uFW and SCS on runtime and solution accuracy. 
	As presented in Table \ref{table: MC}, the computational time of uFW is significantly less than that of SCS. For problems with a smaller $\dt'$ $(\approx 0.5)$, uFW is around $30 \sim 80$ times faster than SCS. 
	For problems with larger $\dt' (\approx 1)$, even though uFW gets slower, it is still about $3\sim 5$ times faster than SCS in runtime. 
	Moreover, upon termination, 
	uFW typically finds a solution with (relative) optimality gap less than $3\times 10^{-4}$, which appears to be more accurate than the solution found by SCS with tolerance $3\times 10^{-3}$ (even though these two solutions are not strictly comparable) -- the latter typically has a \textit{SCS feas.} larger than $3\times10^{-3}$.
Table \ref{table: MC} shows that the ``SCS feas." value is positive (i.e., the solution is not feasible) for all 
instances. This is the reason why some of the SCS gaps reported in Table \ref{table: MC} are negative.

\begin{table}[]
	\centering
	\caption{Comparison of uFW and SCS for solving \eqref{MC-opt2} with $r = r_1 = 5$. Here, ``gap'' refers to relative optimality gap as defined in the text. Note that SCS time is the runtime of SCS with tolerance $10^{-3}$.}
	\label{table: MC}
	\scalebox{0.83}{
		\begin{tabular}{|c|c|ccccc|}
				\hline
				& delta & uFW time & uFW gap & SCS time & SCS gap & SCS feas. \\ \hline 
				\multirow{3}{*}{\begin{tabular}[c]{@{}c@{}}m=700, n=700,\\ nnzr=0.3\end{tabular}}  
				& 0.5   &   18.18   &   9.63e-05  &  694.01    &  7.00e-05   &   4.71e-03    \\ 
				& 0.8   &   90.14   &   5.20e-05  &  785.85    &  6.99e-05   &   4.20e-03    \\ 
				& 1.0   &   194.47   &   1.01e-04  &  793.35    &  1.86e-04   &   5.34e-03     \\ \hline 
				\multirow{3}{*}{\begin{tabular}[c]{@{}c@{}}m=1000, n=500,\\ nnzr=0.3\end{tabular}}  
				& 0.5   &   13.33   &   7.81e-05  &  952.35    &  2.62e-04   &   1.68e-02    \\ 
				& 0.8   &   73.34   &   4.68e-05  &  1172.53    &  1.29e-04   &   6.20e-03    \\ 
				& 1.0   &   180.20   &   8.48e-05  &  1261.50    &  -2.21e-04   &   2.37e-03     \\ \hline 
				\multirow{3}{*}{\begin{tabular}[c]{@{}c@{}}m=300, n=3000,\\ nnzr=0.2\end{tabular}}  
				& 0.5   &   14.38   &   1.56e-04  &  1309.38    &  -6.96e-04   &   3.30e-02    \\ 
				& 0.8   &   243.56   &   2.56e-04  &  1912.09    &  8.20e-04   &   8.97e-03    \\ 
				& 1.0   &   299.84   &   1.88e-04  &  1735.94    &  2.81e-03   &   1.25e-02     \\ \hline 
		\end{tabular}
	}
\end{table}

Table \ref{table:MC-large-instances} presents more instances with larger values of $m$ and $n$ on which SCS fails to output a solution (for a tolerance of $3\times 10^{-3}$, as above). The termination rule for uFW is the same as mentioned above. Due to its mild per-iteration cost, uFW is able to solve these problems approximately within minutes to hours.

\begin{table}[h]
	\centering
	\caption{Running times of uFW for large instances on which SCS would not run.}
	\label{table:MC-large-instances}
	\begin{tabular}{|c|c|c|c|c|c|}
			\hline
			& $\delta$ & uFW time &                                                                                       & $\delta$ & uFW time \\ \hline 
			\multirow{3}{*}{\begin{tabular}[c]{@{}c@{}}m=1000, n=1000, \\ nnzr=0.2\end{tabular}} &    0.5      &   20.64       & \multirow{3}{*}{\begin{tabular}[c]{@{}c@{}}m=3000, n=300, \\ nnzr=0.2\end{tabular}}   &    0.5      &    16.91      \\ 
			&    0.8      &   100.75       &                                                                                       &    0.8      &    122.60      \\ 
			&     1.0     &   245.68       &                                                                                       &     1.0     &    399.60      \\ \hline 
			\multirow{3}{*}{\begin{tabular}[c]{@{}c@{}}m=3000, n=3000, \\ nnzr=0.1\end{tabular}} &    0.5      &   64.68       & \multirow{3}{*}{\begin{tabular}[c]{@{}c@{}}m=7000, n=7000, \\ nnzr=0.05\end{tabular}} &    0.5      &    212.86      \\ 
			&      0.8    &   609.70       &                                                                                       &    0.8      &    3520.89      \\ 
			&      1.0    &   1842.15       &                                                                                       &    1.0      &    7498.59      \\ \hline 
	\end{tabular}
\end{table}

\section{Acknowledgements} The authors would like to thank the AE and the Reviewers for their comments 	that led to improvements in the manuscript.

\appendix

\section{Proof of Proposition \ref{FW-contraction}}\label{sec:appendix-prop} The proof of Proposition \ref{FW-contraction} here closely follows the proof of Proposition 9 in \cite{gutman2019condition} with two major differences: (i)~we apply the analysis to the function $\bar f(\cdot , w)$ for a dynamically changing $w$; and (ii)~instead of the relative strong convexity to the constraint, we utilize a lower bound. Since these results do not appear in~\cite{gutman2019condition}, we present the full proof here.

We first introduce some new notations and establish three auxiliary lemmas (Lemmas~\ref{lemma: tdf-strong-convex}--\ref{square-bound-lemma-CGG}). 
Let $N$ be the number of vertices of $S$, and $\Dt_N$ be the standard simplex in $\R^N$, that is, $\Dt_N = \{x\in \R^N: x_i\ge0~\forall i\in [N],~ \boldsymbol{1}_N^\T x = 1\}$.
First, we introduce an auxiliary function $\wtd f$. Let $ B_1 \in \R^{(n-r)\times N}$ be the matrix whose columns are vertices of $S$, and let
$
B := \bigl[\begin{smallmatrix}
	B_1 & 0 \\ 0 & I_{r}
\end{smallmatrix} \bigr]\in \R^{n\times (N+r)}.
$
Define $\wtd f: \R^{N+r} \rightarrow \R \cup \{\infty\}$ such that $\wtd f:= \bar f \circ B$. Note that $\dom ( \wtd f)\supseteq \Dt_{N} \times \R^r$ and
recall the definition of $\psi(S)$ appearing in~\eqref{facial-distance-defn}.

Before analyzing the Frank-Wolfe steps using the function $\wtd f$, we need a result regarding the strong-convexity of $\wtd f$ in certain directions. For $\wtd u^0 \in \Dt_N$, we define 
$$
Z_{B_1, \Dt_N} (\wtd u^0) := \{  \wtd u \in \Dt_N | ~ B_1 \wtd u = B_1 \wtd u^0   \} \ .
$$
\begin{lemma}\label{lemma: tdf-strong-convex}
	Suppose $f$ satisfies Assumption \ref{awaystep-ass} and $\wtd f$ is as defined above. Define $ \wtd \mu =\mu \psi(S)^2 / 4 $. 
	Then
	for any 
	$(\wtd u^1, w), (\wtd u^2, w) \in \Dt_N \times \R^r$
	satisfying 
	\begin{eqnarray}\label{condition-u1-u2}
		\| \wtd u^2 - \wtd u^1 \|_1 = \inf_{\wtd u \in Z_{B_1, \Dt_N}(\wtd u^2)} \| \wtd u - \wtd u^1 \|_1  \ ,
	\end{eqnarray}
	it holds
	\begin{eqnarray}
		\wtd f(\wtd u^2 , w) \ge 
		\wtd f(\wtd u^1 , w) + \la
		\na_{\wtd u} \wtd f(\wtd u^1,w ), \wtd u^2 - \wtd u^1
		\ra + 
		({\wtd \mu}/{2})  \| \wtd u^2 - \wtd u^1 \|_1^2\ .\nonumber
	\end{eqnarray} 
\end{lemma}
\begin{proof}
	Let $\wtd \mu'$ be the largest possible value such that
	\begin{eqnarray}\label{tdf-strong-convex-eq}
		\wtd f(\wtd u^2 , w) \ge 
		\wtd f(\wtd u^1 , w) + \la
		\na_{\wtd u} \wtd f(\wtd u^1,w ), \wtd u^2 - \wtd u^1
		\ra + 
		(\wtd \mu'/2) \| \wtd u^2 - \wtd u^1 \|_1^2
	\end{eqnarray}
	holds for all $\wtd u^1$ and $\wtd u^2$ satisfying \eqref{condition-u1-u2}.
	For $\wtd u , \wtd v \in \Dt_N$,
	we introduce the notation 
	\begin{eqnarray}
		\| Z_{B_1, \Dt_N} (\wtd v) - \wtd u \|_1 := \inf_{ \wtd v^1 \in Z_{B_1, \Dt_N} (\wtd v) } \| \wtd v^1 - \wtd u\|_1 \ .
		\nonumber
	\end{eqnarray}
	Fix $w\in \R^r$ and consider the functions $ \wtd f_w(\cdot) = \wtd f (\cdot , w) : \Dt_N \rightarrow \R $ and $  \bar f_w(\cdot) = \bar  f (\cdot , w) : S \rightarrow \R $, then it is immediate that $\wtd f_w = \bar f_w \circ B_1$, and \eqref{tdf-strong-convex-eq} is equivalent to
	\begin{eqnarray}
		\wtd f_w (\wtd u^2 ) \ge \wtd f_w (\wtd u^1 ) + \la  \na \wtd f_w(\wtd u^1), \wtd u^2 - \wtd u^1  \ra+ 
		({\wtd \mu'}/{2}) \| \wtd u^2 - \wtd u^1 \|_1^2 \ .
		\nonumber
	\end{eqnarray}
	It then follows from Theorem 1 of \cite{gutman2019condition}
	that
	\begin{eqnarray}\label{ineq20}
		\wtd \mu' \ge \inf_{\mbox{$\scriptsize{
					\begin{array}{c}
						\wtd u, \wtd v \in \Dt_N, \\
						\wtd u\notin Z_{B_1, \Dt_N} (\wtd v)
					\end{array}
				} $} }
		\frac{\mu \| B_1(\wtd u - \wtd v )\|_2^2}{\| Z_{B_1,\Dt_N}(\wtd v) - \wtd u \|_1^2} \ .
	\end{eqnarray}
	Furthermore, 
	by Proposition 1 of \cite{gutman2018condition}, the right hand side of \eqref{ineq20} equals $\mu \psi(S)^2 / 4 = \wtd \mu$. \end{proof}

	\begin{lemma}\label{lemma: simplex vector decompose}
		For any $a,b \in \Dt_N$, there exist $p,q\in \Dt_N$ such that 
		\begin{eqnarray}
			a-b = {\| a-b\|_{1}} (p-q) /2\quad {\rm{and}} \quad {\rm{supp}}(p) \subseteq {\rm{supp}}(a) \ ,
			\nonumber
		\end{eqnarray}
		where ${\rm{supp}}(a)$ and ${\rm{supp}}(p)$ denote the indices of nonzero coordinates of $a$ and $p$ respectively. 
	\end{lemma}
	\begin{proof}
		For any vector $x\in \R^n$, let $x^+$ be the vector in $\R^n$ with $x_i^+= \max \lt\{ x_i,0 \rt\}$ and $x^- =  x^+ - x$.
		Assume $a\neq b$ (otherwise the conclusion is trivial). 
		Let 
		$$p:= 2  (a-b)^+ / \|a-b\|_1 ~~\text{and}~~q:= 2 (a-b)^- / \|a-b\|_1. $$
		Then we have
		$
		a-b = ({\| a-b\|_{1}}/{2}) (p-q)
		$~ and~ $\text{supp}(p) \subseteq \text{supp}(a)$. Note that
		\begin{eqnarray}
			1_N^\T p - 1_N^\T q = \frac{2}{\|a-b\|_1} (1^\T_N a- 1_N^\T b) = 0,~~~~~
			1_N^\T p + 1_N^\T q = \frac{2}{\|a-b\|_1} \|a-b\|_1 = 2 \ .
			\nonumber
		\end{eqnarray}
		As a result, it holds $1_N^\T p = 1_N^\T q =1$, hence $p,q \in \Dt_N$. 
	\end{proof}

	Lemma~\ref{square-bound-lemma-CGG} provides a key inequality that allows us to establish a contraction in optimality gap.
	Recall that we have defined
	$F: \R^{r} \rightarrow \R $ as $F(w) = \rminf_{u \in S} \bar f(u,w) $. 
	
	\smallskip
	
	\begin{lemma}\label{square-bound-lemma-CGG}
		Let $\{(u^k,w^k,d^k)\}_{k\ge0}$ be the iterates generated by Algorithm~\ref{Away-step CG-G (normal form)}. Then
		\begin{equation*}
			\la \na_u \bar f(u^k,w^{k+1}) , d^k \ra^2 \ge 2 \wtd \mu (\bar f(u^k,w^{k+1}) - F(w^{k+1}))\ .
		\end{equation*}
	\end{lemma}
	
	\begin{proof}
		Denote $u^*_{w^{k+1}} \in \argmin_{u\in S} \bar f(u,w^{k+1})$. Let $\{\wtd u^k\}_{k\ge0}$ be a sequence in $\Dt_{N}$ such that $ u^k =  B_1 \wtd u^k  $, and $\wtd u^*_{w^{k+1}} $ be a point in $\R^N$ such that 
		$$\wtd u^*_{w^{k+1}} \in \argmin\nolimits_{\td u \in \Dt_N} 
		\lt\{  \| \wtd u - \wtd u^k \|_1  \ \ | \   B_1 \wtd u =
		u^*_{w^{k+1}}\rt\} .$$
		Recall that $V(u^k)$ is the subset of vertices of $S$ corresponding to the support of $u^k$. 
		Let $\wtd V(u^k)$ be the subset of vertices of $\Dt_N$ corresponding to $V(u^k)$. 
		Then we know that $\wtd u^k = \sum_{\wtd v \in \wtd V(u^k)} \lam_{\wtd v} \wtd v$ where $\lam_{\wtd v} >0$. (We say that $\wtd u^k$ is supported on $\wtd V(u^k)$).
		By Lemma \ref{lemma: simplex vector decompose} with $a = \wtd u^k$ and $b= \wtd u^*_{w^{k+1}}$, there exist $\wtd p, \wtd q \in \Dt_N$ such that 
		\begin{eqnarray}\label{ineq3}
			\wtd u^k - \wtd u^*_{w^{k+1}} = ({\beta}/{2}) (\wtd p - \wtd q ) \ ,
		\end{eqnarray}
		where $\beta := \| \wtd u^k - \wtd u^*_{w^{k+1}} \|_1$, 
		and the support of $\wtd p$ is a subset of $\wtd V(u^k )$. 
		
		As a result,
		if we let $p =  B_1 \wtd p$ and $q =  B_1 \wtd q$, then  we have $p\in \text{conv}(V(u^k))$. In addition, since $1_N^\T \wtd q =1$ and the columns of $B_1$ correspond to the vertices of $S$, we have $q\in S$. Multiplying both sides of equality \eqref{ineq3} by $B_1$, we have
		\begin{equation}\label{eqn: tmp}
			u^k - u^*_{w^{k+1}} = ({\beta}/{2}) (p-q).
		\end{equation}
		With these results at hand, we get
		\begin{equation}\label{ineq4}
			\begin{aligned}
				&~~~~\la \na_u \bar f(u^k, w^{k+1}), u^k - u_{w^{k+1}}^* \ra\\
				&\mathop{=}\limits^{(i)}
				({\beta}/{2})	\la \na_u \bar f(u^k, w^{k+1}), p-q \ra \\
				&\mathop{\le}\limits^{(ii)}
				({\beta}/{2}) \Big(  \max_{u \in \text{conv}(V(u^k))} \la \na_u \bar f(u^k, w^{k+1}), u \ra - \min_{u\in S} \la \na_u \bar f(u^k, w^{k+1}), u \ra   \Big) \\
				&\mathop{=}\limits^{(iii)}
				({\beta}/{2}) \la \na_u \bar f(u^k, w^{k+1}) , v^k - s^k \ra 
			\end{aligned}
		\end{equation}
		where $(i)$ is from \eqref{eqn: tmp};
		$(ii)$ is from the fact $p\in {\rm conv }(V(u^k))$ and $q\in S$; $(iii)$ is from the definitions of $s^k$ and $v^k$ in Algorithm \ref{Away-step CG-G (normal form)}. 
		From the definition of $d^k$ in Algorithm~\ref{Away-step CG-G (normal form)}, we have 
		\begin{eqnarray}
			\la \na_u \bar f(u^k, w^{k+1}), d^k \ra &\le& (1/2)[\la\na_u \bar f(u^k, w^{k+1}) , s^k - u^k  \ra + \la\na_u \bar f(u^k, w^{k+1}) ,  u^k - v^k  \ra] \nonumber\\
			&=& (1/2)\la  \na_u \bar f(u^k, w^{k+1}) , s^k - v^k  \ra \ . \nonumber
		\end{eqnarray}
		Combining with \eqref{ineq4}, we have
		\begin{eqnarray}
			\la \na_u \bar f(u^k, w^{k+1}) , u^k - u^*_{w^{k+1}} \ra \le
			\frac{\beta}{2} \la \na_u \bar f(u^k, w^{k+1}) , v^k -s^k \ra
			\le - \beta \la \na_u \bar f(u^k, w^{k+1}), d^k \ra \ . \nonumber
		\end{eqnarray}
		On the other hand, from the strong convexity of $\wtd f$ in the first block (Lemma \ref{lemma: tdf-strong-convex}), we have 
		\begin{eqnarray}\label{ineq6}
			F(w^{k+1}) = \wtd f(\wtd u^*_{w^{k+1}} , w^{k+1}) &\ge& \wtd f(\wtd u^k, w^{k+1}) + \la \na_{\wtd u} \wtd f(\wtd u^k, w^{k+1}) , \wtd u^*_{w^{k+1}} - \wtd u^k \ra + \wtd \mu \beta^2/2 \nonumber\\
			& \mathop{=}\limits^{(i)}& 
			\bar f(u^k,w^{k+1} ) + \la \na_u \bar f(u^k, w^{k+1}), u^*_{w^{k+1}} - u^k \ra +  \wtd \mu \beta^2/2 \nonumber\\ &\mathop{\ge}& 
			\bar f(u^k,w^{k+1} )+ \beta  \la \na_u \bar f(u^k, w^{k+1}), d^k \ra +  \wtd \mu \beta^2/2  \ ,
		\end{eqnarray}
		where $(i)$ is because $\wtd f(\wtd u^k, w^{k+1}) = \bar f(B_1 \wtd u^k , w^{k+1}) = \bar f(u^k, w^{k+1})$ and
		\begin{align*}
			&\na_{\wtd u} \wtd f(\wtd u^k, w^{k+1}) = \na_{\wtd u} \bar f(B_1 \wtd u^k , w^{k+1}) = B_1^\T \na_{u} \bar f(u^k, w^{k+1}) \nonumber\\
			\Longrightarrow~~~&
			\la \na_{\wtd u} \wtd f(\wtd u^k, w^{k+1}) , \wtd u^*_{w^{k+1}} - \wtd u^k \ra = 	\la \na_u \bar f(u^k, w^{k+1}), B_1 \wtd u^*_{w^{k+1}} - B_1 \wtd u^k \ra \\
			& ~~~~~~~~~~~~~~= \la \na_u \bar f(u^k, w^{k+1}),  u^*_{w^{k+1}} -  u^k \ra \ .
		\end{align*}		
		Therefore, from \eqref{ineq6}, one has
		\begin{eqnarray}
			-	\la \na_u \bar f(u^k, w^{k+1}), d^k \ra &\ge& 
			(1/\beta) (\bar f(u^k, w^{k+1}) - F(w^{k+1})) + {\wtd \mu \beta}/{2}  \nonumber\\
			&\ge& 
			\sqrt{ 2\wtd \mu (\bar f(u^k, w^{k+1}) - F(w^{k+1}))  } \nonumber 
		\end{eqnarray}
		where the last step is by Cauchy-Schwarz inequality. 
	\end{proof}

Now we are ready to prove Proposition \ref{FW-contraction}:

\textbf{Proof of Proposition \ref{FW-contraction}}:
We prove Proposition \ref{FW-contraction} by discussing two different cases. 

\smallskip

\noindent
\textbf{(Case a)} 
This case includes two subcases: 

(Case a.1) $\al_k < \al_{\max}   $;

(Case a.2) $\al_k = \al_{\max}$, and iteration $k$ takes a FW step. 

When either (Case a.1) or (Case a.2) happens, we have $|V(u^{k+1}) | \le |V(u^k)|+1$. Furthermore, we claim that 
\begin{equation}\label{case-a-claim}
	\bar f(u^{k+1},w^{k+1}) \le  \min_{\al \in [0,1]} \lt\{ \bar f(
	u^k + \al d^k , w^{k+1}) \rt\}. 
\end{equation}
Indeed, when (Case a.2) happens, then $\al_{\max}=1$, so by the updating rule, \eqref{case-a-claim} holds true. When (Case a.1) happens, we have 
\begin{equation}
	\bar f(u^{k+1},w^{k+1}) = \min_{\al \in [0,\al_{\max}]} \lt\{ \bar f(
	u^k + \al d^k , w^{k+1}) \rt\} = \min_{\al \in [0,\infty)} \lt\{ \bar f(
	u^k + \al d^k , w^{k+1}) \rt\} \nonumber
\end{equation}
where, the first equality is from the updating rule of the algorithm, and the second equality is because $\al_k < \al_{\max}$ and $f$ is convex. 
As a result, claim \eqref{case-a-claim} holds true. 

From \eqref{case-a-claim} and the definition of $\bar C^S_{f,x^0}$ we have 
\begin{eqnarray}\label{ineq-case--a1}
	&& \bar f(u^{k+1},w^{k+1}) \nonumber\\
	&\le&
	\min_{\al \in [0,1]} \Big\{ \bar f(u^k,w^{k+1}) + \al \la \na_u \bar f(u^k,w^{k+1}) , d^k \ra + \bar C_{f,x^0}^S\al^2/2
	\Big\} \ .
\end{eqnarray}
Let $\bar \al_k$ be the optimal solution in \eqref{ineq-case--a1}. 
If $\bar \al_k <1$, then \eqref{ineq-case--a1} implies 
\begin{equation}\label{ineq-case--a2}
	\begin{aligned}
		\bar f(u^{k+1},w^{k+1}) &\le \bar f(u^k,w^{k+1})  - {\la \na_u \bar f(u^k,w^{k+1}) ,d^k \ra^2}/({2 \bar C_{f,x^0}^S}) \\
		&\le
		\bar f(u^k,w^{k+1})  - {\wtd \mu} \lt(\bar f(u^k,w^{k+1})  - F(w^{k+1})\rt) / \bar C_{f,x^0}^S
	\end{aligned}
\end{equation}
where the second inequality is by Lemma \ref{square-bound-lemma-CGG}. 
Recall that $\ga = \mu \psi(S)^2/(4\bar C^S_{f,x^0}) = \wtd \mu/\bar C^S_{f,x^0}$. 
As a result, we have
\begin{equation}\label{case-a--conclusion1}
	\bar f(u^{k+1},w^{k+1})  - f^* \le \lt( 1- \ga\rt) \lt( \bar f(u^k,w^{k+1})  - f^* \rt) + 
	\ga \lt(F(w^{k+1}) - f^*\rt).
\end{equation}

If $\bar \al_k =1$, then by taking derivative w.r.t $\al$ in \eqref{ineq-case--a1}, we know that
$$
- \la \na_u \bar f(u^k,w^{k+1}), d^k \ra \ge \bar C_{f,x^0}^S.
$$
Combining this with \eqref{ineq-case--a1} again we have 
\begin{equation}\label{eq2}
	\begin{aligned}
		\bar f(u^{k+1},w^{k+1}) &\le \bar f(u^k,w^{k+1}) + \min_{\al \in [0,1]} \lt\{  (\al - \al^2 /2) \la \na_u \bar f(u^k,w^{k+1}),d^k  \ra \rt\}  \\ 
		& =
		\bar f(u^k,w^{k+1}) + (1/2) \la \na_u \bar f(u^k,w^{k+1}),d^k  \ra \ ,
	\end{aligned}
\end{equation}
where the last equality is because $\la \na_u \bar f(u^k,w^{k+1}),d^k  \ra \le 0 $. 
On the other hand, 
\begin{equation}\label{ineq7}
	\begin{aligned}
		\la \na_u \bar f(u^k,w^{k+1}) , u^*_{w^{k+1}} - u^k \ra 
		&\mathop{\ge}\limits^{(i)} \la \na_u \bar f(u^k,w^{k+1})  ,s^k - u^k  \ra  \\
		& \mathop{\ge}\limits^{(ii)}
		\la \na_u \bar f(u^k,w^{k+1}) ,d^k  \ra
	\end{aligned}
\end{equation}
where $(i)$ is from the definition of $s^k $, and $(ii)$ is from the definition of $d^k$ in Algorithm~\ref{Away-step CG-G (normal form)}. 
Combining \eqref{ineq7} and \eqref{eq2}, we have
\begin{equation}
	\begin{aligned}
		\bar f(u^{k+1},w^{k+1} ) &\le 
		\bar f(u^k,w^{k+1})+ (1/2) \la \na_u \bar f(u^k,w^{k+1}), u^*_{w^{k+1}} - u^k \ra \\
		& \le 
		\bar f(u^k,w^{k+1})+  (1/2) \lt(F(w^{k+1}) - \bar f(u^k,w^{k+1})\rt) , 
	\end{aligned}
	\nonumber
\end{equation}
where the last inequality is from the convexity of $f$. As a result, we arrive at
\begin{equation}\label{case-a--conclusion2}
	\bar f(u^{k+1},w^{k+1}) - f^* \le 
	(1/2) \lt[ \bar f(u^k,w^{k+1}) - f^*  \rt] + 
	(1/2) \lt[F(w^{k+1})- f^*  \rt] \ .
\end{equation}
Recall that $\rho = \min\{\ga, 1/2\}$ and $F(w^{k+1}) \le \bar f(u^{k+1},w^{k+1})$, so combining \eqref{case-a--conclusion1} and \eqref{case-a--conclusion2} we have 
\begin{equation}\label{case-a--conclusion3}
	\bar f(u^{k+1},w^{k+1})  - f^* \le \lt( 1- \rho\rt) \lt( \bar f(u^k,w^{k+1})  - f^* \rt) + 
	\rho \lt(F(w^{k+1}) - f^*\rt) \ .
\end{equation}

\noindent
\textbf{(Case b):} $\al_k = \al_{\max}$ and iteration $k$ takes an away step. In this case, we have $ |V(u^{k+1})| \le |V(u^k)| -1 $. Since $|V(u^0)| = 1$ and $|V(u^k)| \ge 1$ for $k\ge1$, it follows that for each iteration when (Case b) occurs, there must have been at least one earlier iteration index at which  (Case a) occured. Hence in the first $k$ iterations, (Case b) occurs at most $k/2$ times.

Combining the discussions in (Case a) and (Case b), we reach the conclusion.

	\section{Proof of Proposition \ref{prop: stop-criterion-conv2}}\label{appendix sec: proof of prop-stop-criterion-conv2}
	First, by \eqref{combine2} (this inequality also holds for iterations generated by Algorithm \ref{CG-G with away steps}), we know 
	\begin{equation}
		\| \cP_T \na f(x^k) \|_2 \le \sqrt{(2/\eta) (f(x^k) - f(y^k))} \le 
		\sqrt{(2/\eta) (f(x^k) - f^* )}. 
	\end{equation}
	Combining the above with \eqref{ineq--new7} (and recalling, $\wtd H_k = \| \cP_{T} \na f(x^k) \|_2$), we have 
	\begin{equation}
		H_k \le (1+ \eta L^T_{f,x^0})\| \cP_T \na f(x^k)\|_2 \le (1+ \eta L^T_{f,x^0}) \sqrt{(2/\eta) (f(x^k) - f^* )}. 
	\end{equation}
	This completes the proof of   Part~(1). 
	
	To prove Part~(2), 
	similar to the proof of Theorem \ref{CGG-linear-conv}, we reduce the arguments to the axis-aligned case (Algorithm \ref{Away-step CG-G (normal form)}). Note that in this case, $x^k = (u^k, w^k)$ and $y^k = (u^k, w^{k+1})$. In the proof of Proposition \ref{FW-contraction}, we show that in the first $k$ iterations, there are at least $k/2$ iterations such that (Case a) happens. If (Case a) happens with $\bar \al_k<1$, then we have
	\begin{equation}\label{Gk-bound1}
		G_k \le - \la  \na f(y^k) , d^k \ra \le \sqrt{ 2\bar C_{f,x^0}^S (f(y^k) - f(x^{k+1})) } 
		\le \sqrt{ 2\bar C_{f,x^0}^S (f(y^k) - f^*) } 
	\end{equation}
	where the first inequality is by the definition of $d^k$, and the second is by the first inequality in \eqref{ineq-case--a2}. If (Case a) happens with $\bar \al_k = 1$, then we have 
	\begin{equation}\label{Gk-bound2}
		G_k \le - \la  \na f(y^k) , d^k \ra \le 2(f(y^k) - f(w^{k+1})) \le 2(f(y^k) - f^*) 
	\end{equation}
	where the second inequality is by \eqref{eq2}. Combining \eqref{Gk-bound1} and \eqref{Gk-bound2}, the proof of Part~(2) is complete.

	\bibliographystyle{plain}
\bibliography{mybib_FW_unbounded_cons2}

\end{document}